\documentclass[a4paper,reqno]{amsart}

\usepackage{amsfonts}
\usepackage{amsmath,amssymb,amsthm,amsthm}
\usepackage{mathtools}
\usepackage{etoolbox}
\usepackage{mathrsfs}
\usepackage[english]{babel}
\usepackage[utf8]{inputenc}
\usepackage{hyperref}
\usepackage{microtype}
\usepackage{fullpage}
\usepackage[leftcaption]{sidecap}
\sidecaptionvpos{figure}{m}

\DisableLigatures{encoding = *, family = * }

\newtheorem{theorem}{Theorem}[section]
\newtheorem{definition}[theorem]{Definition}
\newtheorem{lemma}[theorem]{Lemma}
\newtheorem{proposition}[theorem]{Proposition}
\newtheorem{remark}[theorem]{Remark}

\numberwithin{equation}{section}

\newcommand{\RR}{\mathbb{R}}
\newcommand{\NN}{\mathbb{N}}
\newcommand{\norm}[2]{\left\|#1\right\|_{#2}}

\newcommand{\intr}[1]{\overset{\circ}{#1}}
\newcommand{\ptml}{p_{T,ml}^\ast}
\newcommand{\ptmld}{\tilde{p}_{T,ml}^\ast}
\newcommand{\pml}{p_{ml}^\ast}
\newcommand{\pmld}{\tilde{p}_{ml}^\ast}
\newcommand{\pt}{p_{T,2}^\ast}
\newcommand{\uml}{u_{ml}^\ast}
\newcommand{\uu}{u_2^\ast}

\begin{document}
	
\title{Multilevel control by duality}

\author{Umberto Biccari\textsuperscript{\,$\ast$}}  
\address{\textsuperscript{$\ast$}\, [1] Chair of Computational Mathematics, Fundaci\'on Deusto, Avenida de las Universidades 24, 48007 Bilbao, Basque Country, Spain} 
\address{[2]\,Facultad de Ingenier\'ia, Universidad de Deusto, Avenida de las Universidades 24, 48007 Bilbao, Basque Country, Spain.}
\email{umberto.biccari@deusto.es, u.biccari@gmail.com}

\thanks{This project has received funding from the European Research Council (ERC) under the European Union’s Horizon 2020 research and innovation programme (grant agreement NO: 694126-DyCon). The work of both authors is supported by the Grant MTM2017-92996-C2-1-R COSNET of MINECO (Spain), by the Elkartek grant KK-2020/00091 CONVADP of the Basque government and by the Air Force Office of Scientific Research (AFOSR) under Award NO: FA9550-18-1-0242. The work of E.Z. is funded by the Alexander von Humboldt-Professorship program, the European Unions Horizon 2020 research and innovation programme under the Marie Sklodowska-Curie grant agreement No.765579-ConFlex and the Transregio 154 Project ‘‘Mathematical Modelling, Simulation and Optimization Using the Example of Gas Networks’’, project C08, of the German DFG} 

\author{Enrique Zuazua\textsuperscript{\,$\ddagger$}}
\address{\textsuperscript{$\ddagger$}\, [1] Chair for Dynamics, Control and Numerics, Alexander von Humboldt-Professorship, Department of Data Science, Friedrich-Alexander-Universit\"at Erlangen-N\"urnberg, 91058 Erlangen, Germany.}
\address{[2] Chair of Computational Mathematics, Fundaci\'on Deusto, Avenida de las Universidades 24, 48007 Bilbao, Basque Country, Spain.} 
\address{[3] Departamento de Matem\'aticas, Universidad Aut\'onoma de Madrid, 28049 Madrid, Spain.}
\email{enrique.zuazua@fau.de}

\begin{abstract}
We discuss the multilevel control problem for linear dynamical systems, consisting in designing a piece-wise constant control function taking values in a finite-dimensional set. In particular, we provide a complete characterization of multilevel controls through a duality approach, based on the minimization of a suitable cost functional. In this manner we build optimal multi-level controls and characterize the time needed for a given ensemble of levels to assure the controllability of the system. Moreover, this method leads to efficient numerical algorithms for computing multilevel controls. 
\end{abstract}


\maketitle 

\section{Introduction}\label{sec:intro} 

In this paper, we discuss the multilevel control problem for linear dynamical systems, consisting in designing a piece-wise constant control function taking values in a finite-dimensional set. As we shall see, this control notion can be understood as a generalization of the well-known concept of bang-bang controls, which have been widely studied in the literature. As a matter of fact, inspired by the existing literature on bang-bang controls (see for instance \cite{fabre1995approximate,micu2012time,micu2004introduction}), we show how multilevel controls can be designed by an optimal control methodology, allowing to characterize important structural properties such as the time needed to control a given dynamical system. This leads to efficient computational tools to build optimal multilevel controls. 

To simplify our presentation, we focus on finite-dimensional ODE systems with scalar controls $u\in\RR$, fulfilling the Kalman rank condition for controllability. 

Our approach is based on the so-called adjoint methodology, which has been systematically associated to optimal control problems (\cite{herty2015numerical}) and their applications to several fields of science and engineering such as aerodynamics (\cite{carpentieri2007adjoint,castro2007systematic,giles2000introduction}), inverse design (\cite{monge2020sparse,morales2019adjoint}), robotics (\cite{colombo2019optimal,vossen2006l1}), non-local and anomalous diffusion (\cite{biccari2021optimal}), and the control of chemical processes (\cite{miranda2008integrating}). In this contribution, motivated by practical applications in power electronics (see \cite{oroya2021multilevel}), we develop a complete analytical theory to build multilevel controls and design efficient numerical tools to approximate them. 

By means of the adjoint methodology, we will show that, for any linear finite dimensional system satisfying the Kalman condition, controllability with multilevel controls may be achieved. Nevertheless, in certain situations that we will detail later, this will require some restrictions on the time horizon or the initial datum that we want to control. As we will see, these restrictions, which may appear in counter-trend with the known controllability results for linear finite-dimensional systems (see \cite{micu2004introduction}), are actually natural in our context and are related with structural constraints of the multilevel control strategy. This is in analogy with the known results on constrained controllability for PDE (see, for instance, \cite{antil2019controllability,biccari2020controllability,loheac2017minimal,pighin2018controllability,pighin2019controllability}).

As for the multilevel nature of the controls, this will be inherited by the introduction of a special penalization in the dual cost functional subject to the adjoint dynamics, constructed as the piece-wise linear interpolation of a given strictly convex function. Moreover, by properly designing this piece-wise linear penalization, we will be able to modulate several properties of multilevel controls such as the number of levels (i.e. the number of different constant values that the control may assume) and their amplitude. 

The present paper is organized as follows: in Section \ref{sec:problem}, we provide the mathematical background for the problem we are going to study. In particular, we introduce the notion of multilevel control and the adjoint methodology that we shall employ. We also present there our main results concerning the design of multilevel controls through duality. Section \ref{sec:proof} is devoted to the proofs of our main results, while in Section \ref{sec:structural} we discuss some structural properties of multilevel controls which can be fully described by our duality argument. In Section \ref{sec:numerics}, we present some numerical simulations showing that our adjoint methodology indeed allows to compute multilevel controls for linear finite-dimensional dynamical systems. Finally, in Section \ref{sec:conclusions}, we gather our conclusions and present some open problems related to our work.

\section{Problem formulation and main result}\label{sec:problem} 
Let $T>0$, $x_0\in\RR^N$, and $A\in\RR^{N\times N}$ and $B\in\RR^N$ be given. Consider the linear finite-dimensional control problem 
\begin{align}\label{eq:primalSystem}
	\begin{cases}
		x'(t) = Ax(t) + Bu(t), & t\in(0,T)
		\\
		x(0) = x_0
	\end{cases}
\end{align}
where $x(\cdot):[0,T]\to\RR^N$ represents the state and $u(\cdot):[0,T]\to\RR$ is the control. For simplicity, we assume $u$ to be scalar.

In this work, we analyze the controllability problem for \eqref{eq:primalSystem}, consisting in finding a control function $u$ in some suitable functional space, such that the corresponding solution with initial datum $x_0$ matches some prescribed target $x_T\in\RR^N$ at time $T$:
\begin{align}\label{eq:controllability}
	x(T) = x_T.
\end{align}

In particular, we are interested in the characterization of \textit{multilevel controls} by means of the so-called \textit{adjoint methodology}.

\begin{definition}\label{def:multilevel}
Let $\mathcal R$ denote the finite-dimensional set 
\begin{align}\label{eq:setR}
	\mathcal R:= \{\rho_1,\ldots,\rho_L\}\subset\RR, \quad L\geq 2.
\end{align}	
We call \textit{multilevel control} any piece-wise constant function $u\in L^\infty(0,T;\mathcal R)$, taking values on $\mathcal R$ with finitely-many jumps, such that the corresponding solution of \eqref{eq:primalSystem} satisfies \eqref{eq:controllability}.
\end{definition}

Since we are working in a finite-dimensional ODE setting, we know that the controllability of \eqref{eq:primalSystem} is equivalent to the \textit{Kalman rank condition}
\begin{align}\label{eq:Kalman}
	\mbox{rank}\Big(B|AB|A^2B|\ldots|A^{N-1}B\Big) = N.
\end{align}

Hence, from now on, we will always assume that the pair $(A,B)$ satisfies \eqref{eq:Kalman}. Moreover, for simplicity, we will focus on the null controllability problem 
\begin{align}\label{eq:NullControllability}
	x(T) = 0.
\end{align}

Recall that, in the finite-dimensional setting that we are considering, \eqref{eq:controllability} and \eqref{eq:NullControllability} are equivalent notions.

The literature on controllability for linear systems like \eqref{eq:primalSystem} is nowadays very extended. In particular, it is well-known that several classes of controls can be built by a duality argument involving the adjoint dynamics
\begin{align}\label{eq:adjoint}
	\begin{cases}
		-p'(t) = A^\top p(t), & t\in(0,T)
		\\
		p(T) = p_T\in\RR^N,
	\end{cases}
\end{align}
where $A^\top$ denotes the transposed of $A$.

The most common situations are $L^2$-controls (see for instance \cite{micu2004introduction}), $L^1$(sparse)-controls in the form of a linear combination of Dirac deltas (see \cite{casas2015sparse,li2014heat}) or $L^\infty$(bang-bang)-controls, i.e. piece-wise constant functions which take only two possible values (see \cite{fabre1995approximate,micu2012time,micu2004introduction}). 

In this work, we show how the adjoint methodology can be adapted in order to build multilevel controls for \eqref{eq:primalSystem} and study some of their structural properties. Notice that, according to Definition \ref{def:multilevel}, such kind of controls are given by 
$u_{ml}\in L^\infty(0,T;\mathcal R)$ of the form
\begin{align}\label{eq:uExpl}
	u_{ml}(t) = \sum_{k=0}^K s_k\chi_{(t_k,t_{k+1})} (t), \quad \NN\ni K<+\infty 
\end{align}
for some $\mathcal S = \{s_k\}_{k=0}^K$ satisfying
\begin{align*}
	s_k\in\mathcal R\; \text{ and } \; s_k\neq s_{k+1}, \quad \mbox{ for all } k\in \{0,\ldots,K\}
\end{align*}
and $\mathcal T = \{t_k\}_{k=1}^{K+1}$ such that
\begin{align*}
	\bigcup_{k=1}^K (t_k,t_{k+1}) = (0,T).
\end{align*}

We shall refer to the sequences $\mathcal S$ and $\mathcal T$ as the \textit{waveform} and \textit{switching points} of the multilevel control. 
Besides, we say that the waveform satisfies the \emph{staircase property} if the control $u_{ml}$ defined in \eqref{eq:uExpl} only switches among consecutive values in $\mathcal R$ (see Figure \ref{fig:staircase}). This can be expressed in rigorous mathematical notation as follows. 

\begin{definition}\label{def:staircaseProp}
We say that a multilevel control $u_{ml}$ of the form \eqref{eq:uExpl} fulfills the \emph{staircase property} if its waveform $\mathcal S$ satisfies
\begin{align}\label{eq:staircaseProp}
	(s_k^{min},s_k^{max}) \cap \mathcal R = \emptyset, \quad \mbox{ for all } k\in \{0,\ldots,K-1\},
\end{align}
where $s_k^{min}:= s_k\wedge s_{k+1}$ and $s_k^{max}:= s_k \vee s_{k+1}$.
\end{definition}

\begin{figure}[h]
	\centering
	\includegraphics[scale=0.35]{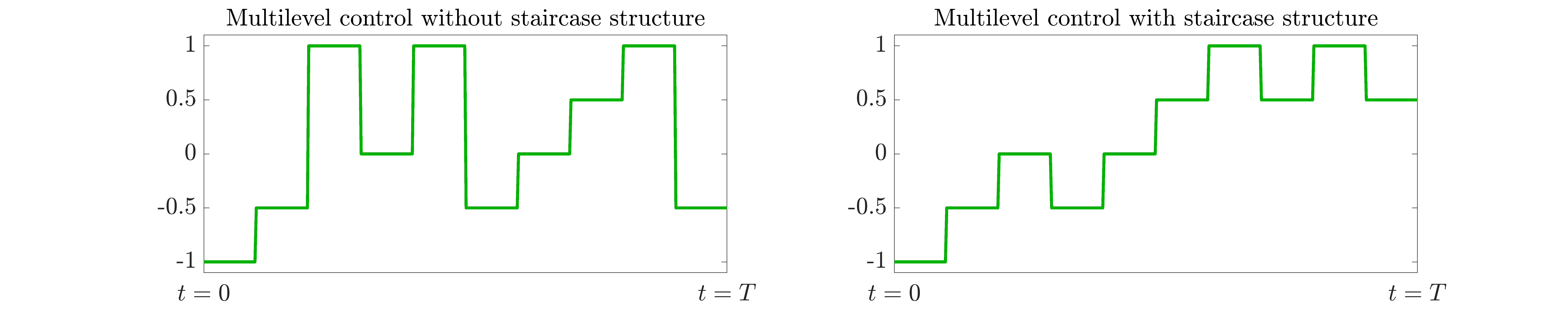}
	\caption{Examples of multilevel controls built over the set $\mathcal R = \{-1,-0.5,0,0.5,1\}$. The control on the left has not the staircase structure and jumps between values of $\mathcal R$ which are not consecutive (for instance from $-0.5$ to $1$). The control on the right, instead, has the staircase structure and only jumps between consecutive values of $\mathcal R$.}\label{fig:staircase}
\end{figure}

The main contribution of this work is to show that multilevel controls enjoying the staircase property \eqref{eq:staircaseProp} can be characterized through the adjoint methodology by solving a suitable optimal control problem. To this end, we will consider two specific situations.

\paragraph*{Case 1: conservative or dissipative dynamics}
We will start by analyzing the cases in which the free dynamics in \eqref{eq:primalSystem} is conservative or dissipative. This corresponds to considering matrices $A$ that, apart from fulfilling the Kalman rank condition \eqref{eq:Kalman}, satisfy one of the following two assumptions:
\begin{align}\label{eq:Acons}
	\mbox{conservative dynamics: } A = -A^\top, \tag{{\bf{H1}}}
\end{align}
	$A^\top$ being the transpose of $A$, or
\begin{align}\label{eq:Adiss}
	\mbox{dissipative dynamics: } \rho(A) < 0, \tag{{\bf{H2}}}
\end{align}
where $\rho(A)$ denotes the spectral radius. 

In both scenarios, provided that the time horizon $T$ is large enough, we can obtain multilevel controls for \eqref{eq:primalSystem} which, in addition, enjoy the staircase property \eqref{eq:staircaseProp}, by solving the optimal control problem
\begin{equation}\label{eq:MultilevelFunct}
	\begin{array}{l}
		\displaystyle \ptml = \min_{\underset{p\text{ solves }\eqref{eq:adjoint}}{p_T\in\RR^N}} J_{ml}(p_T) 
		\\[20pt]
		\displaystyle J_{ml}(p_T) = \int_0^T \mathcal L(B^\top p(t))\,dt + \langle x_0,p(0)\rangle_{\RR^N},
	\end{array} 
\end{equation}
upon a suitable choice of the penalization function $\mathcal L$. 

This large enough controllability time, which may appear in counter-trend with the fact that finite-dimensional dynamics fulfilling the Kalman rank condition are expected to be controllable in any positive time, is needed due to the linear growth of $\mathcal L$ to ensure the coercivity of $J_{ml}$ and, therefore, the existence of a minimizer. We will give more details on this point in the next section.

As for the penalization $\mathcal L$, it will be constructed as a piece-wise linear and continuous interpolation of a given strictly convex function. The reason of this particular choice relies on the simple observation that, as it is classical in mathematical control theory, optimal controls are characterized through the Euler-Lagrange equations associated with the functional one is minimizing. In the case of $J_{ml}$ defined in \eqref{eq:MultilevelFunct}, we will see that  
\begin{align*}
	u_{ml}^\ast(t)\in\partial\mathcal L(B^\top \pml(t)),
\end{align*}
where $\partial\mathcal L$ denotes the sub-differential of the (non-differentiable) function $\mathcal L$, and that, if $\mathcal L$ is piece-wise linear, $u_{ml}^\ast$ will be piece-wise constant. Moreover, it will enjoy the staircase property \eqref{eq:staircaseProp}, as a consequence of the continuity of the solution to the adjoint equation.

\paragraph*{Case 2: general dynamics} For general dynamics which are neither purely conservative nor purely dissipative (such as for example expansive dynamics corresponding to a matrix $A$ such that $\rho(A)>0$), the functional $J_{ml}$ might not be coercive and, therefore, the optimal control problem \eqref{eq:MultilevelFunct} might not have a solution. In these situations, we will see that staircase multilevel controls for \eqref{eq:primalSystem} can still be constructed by considering a slightly modified optimal control problem  
\begin{equation}\label{eq:MultilevelFunctFabre}
	\begin{array}{l}
		\displaystyle \ptmld = \min_{\underset{p\text{ solves }\eqref{eq:adjoint}}{p_T\in\RR^N}} \mathcal J_{ml}(p_T) 
		\\[20pt]
		\displaystyle \mathcal J_{ml}(p_T) = \frac 12\left(\int_0^T \mathcal L(B^\top p(t))\,dt\right)^2 + \langle x_0,p(0)\rangle_{\RR^N},
	\end{array} 
\end{equation}
with the same penalization $\mathcal L$ as in \eqref{eq:MultilevelFunct}. 

As a matter of fact, this minimization process \eqref{eq:MultilevelFunctFabre} would actually yield to multilevel controls for \eqref{eq:primalSystem} in any positive time $T>0$. Indeed, by computing the Euler-Lagrange equation associated with \eqref{eq:MultilevelFunctFabre}, we will see that the multilevel controls in this case are characterized by 
\begin{align}\label{eq:controlMlFabre}
	\tilde{u}_{ml}^\ast(t)\in\Lambda_{T,ml}\partial\mathcal L(B^\top \pmld(t)) \quad\mbox{ with }\quad \Lambda_{T,ml}:=\int_0^T \partial\mathcal L(B^\top \pmld(t))\,dt,
\end{align}
where the factor $\Lambda_{T,ml}$ regulates the intensity of the control allowing to achieve null controllability in any given time horizon $T$. 

\medskip 

Let us now give the precise construction of the penalization $\mathcal L$ that we shall employ both in \eqref{eq:MultilevelFunct} and in \eqref{eq:MultilevelFunctFabre}. To this end, let $\mathcal I = [\varpi_1,\varpi_2]\subset\RR$, with $\varpi_1<\varpi_2$ denoting some closed interval of the real line, and let $\mathcal P\in C^2(\mathcal I)$ be a given non-negative and strictly convex function, attaining its minimum at one of the points in $\mathcal U$. On $\mathcal I$, we introduce a partition $\mathcal U$ defined as
\begin{align}\label{eq:partitionU}
	&\mathcal U = \{u_1,\ldots,u_{M+1}\}, \quad \NN\ni M\geq 2
	\\
	&u_1 = \varpi_1, \;\;u_{M+1} = \varpi_2 \;\mbox{ and }\; u_k<u_{k+1}, \;\mbox{ for all } k\in \{1,\ldots,M\}, \notag
\end{align}
and denote 
\begin{align}\label{eq:hDef}
	h_k:= u_{k+1}-u_k\;\mbox{ for all } k\in\{1,\ldots,M\}\quad\quad\mbox{ and }\quad\quad h:=\max_{k\in\{1,\ldots,M\}}h_k.
\end{align}
For all $k\in\{1,\ldots,M\}$, let  
\begin{align}\label{eq:lambdaK}
	\lambda_k(u):= \frac{(u-u_k)\mathcal P(u_{k+1})+(u_{k+1}-u)\mathcal P(u_k)}{u_{k+1}-u_k}.
\end{align}
Then, we define
\begin{align}\label{eq:Lml}
	\mathcal L(u):= \begin{cases} \lambda_1(u) & \mbox{ if } u<u_1 \\ \lambda_k(u) & \mbox{ if }u\in [u_k,u_{k+1}], \quad k\in\{1,\ldots,M\} \\ \lambda_M(u), & \mbox{ if } u > u_{M+1} \end{cases}.
\end{align}
For the sake of clarity, we display in Figure \ref{fig:MultilevelEx} a particular example of an admissible penalization $\mathcal L$. 
\begin{SCfigure}[0.6][h]
	\includegraphics[scale=0.2]{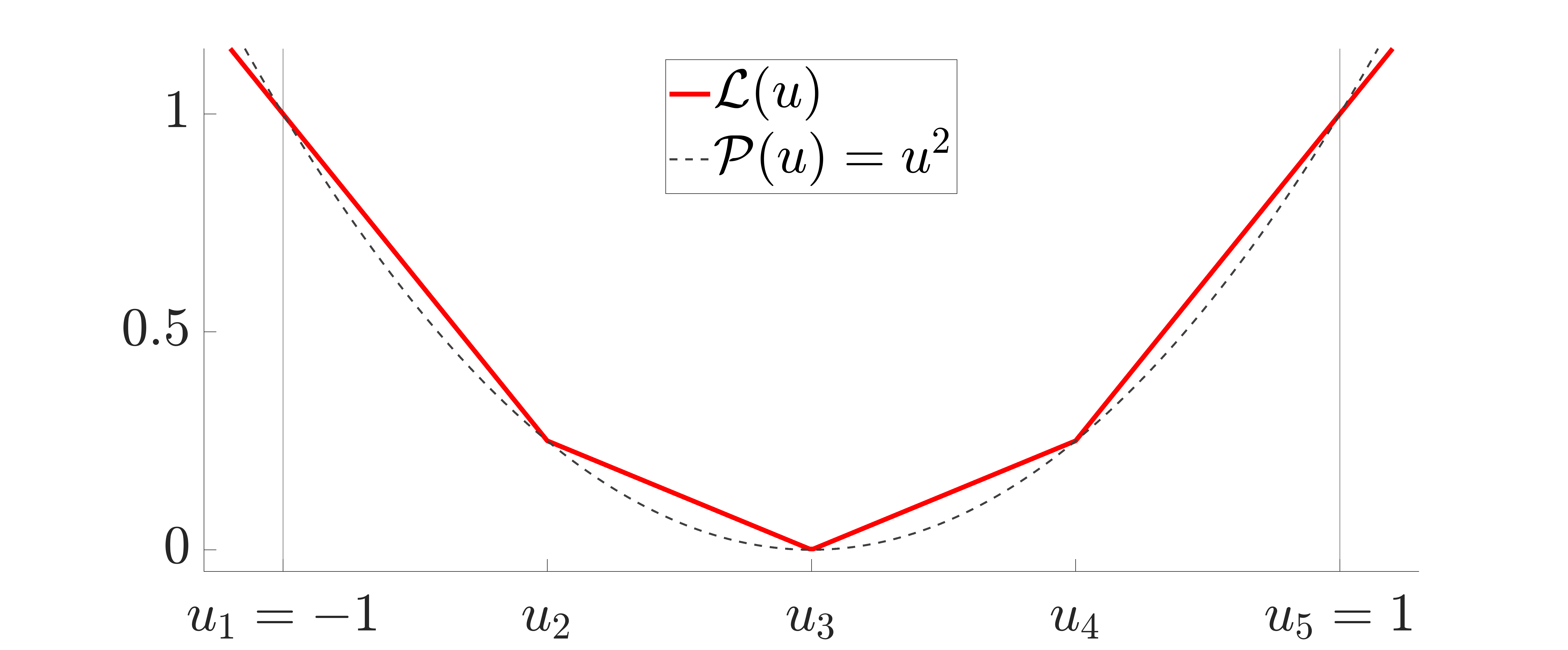}
	\caption{Example of a suitable penalization $\mathcal L$ for multilevel control, constructed interpolating $\mathcal P(u)=u^2$ on the set $\mathcal U = \{-1,-0.5,0,0.5,1\}$.}\label{fig:MultilevelEx}
\end{SCfigure}

Moreover, in view of the employment of \eqref{eq:MultilevelFunct} and \eqref{eq:MultilevelFunctFabre} in the design of multilevel controls, let us notice that, once the penalization $\mathcal L$ is constructed according to \eqref{eq:lambdaK}-\eqref{eq:Lml}, the following properties are satisfied.

\begin{itemize}
	\item[1.] By construction, it is possible to give some upper and lower barriers for $\mathcal L$ in terms of the absolute value function. In particular, there exists two positive constants $0<\alpha_1<\alpha_2<+\infty$, possibly depending on $\mathcal L$, such that for all $u\in\mathcal I$ (see Figure \ref{fig:FunctCompare})
	\begin{align}\label{eq:compare}
		\alpha_1|u|\leq\mathcal L(u)\leq\alpha_2|u|.
	\end{align}
	This property will eventually lead to the coercivity of the functional $J_{ml}$ in \eqref{eq:MultilevelFunct}, provided that the time horizon $T$ is large enough depending on the state we aim to control, and $\mathcal J_{ml}$ \eqref{eq:MultilevelFunctFabre} for any $T>0$. 
	\begin{SCfigure}[0.5][h]
		\includegraphics[scale=0.2]{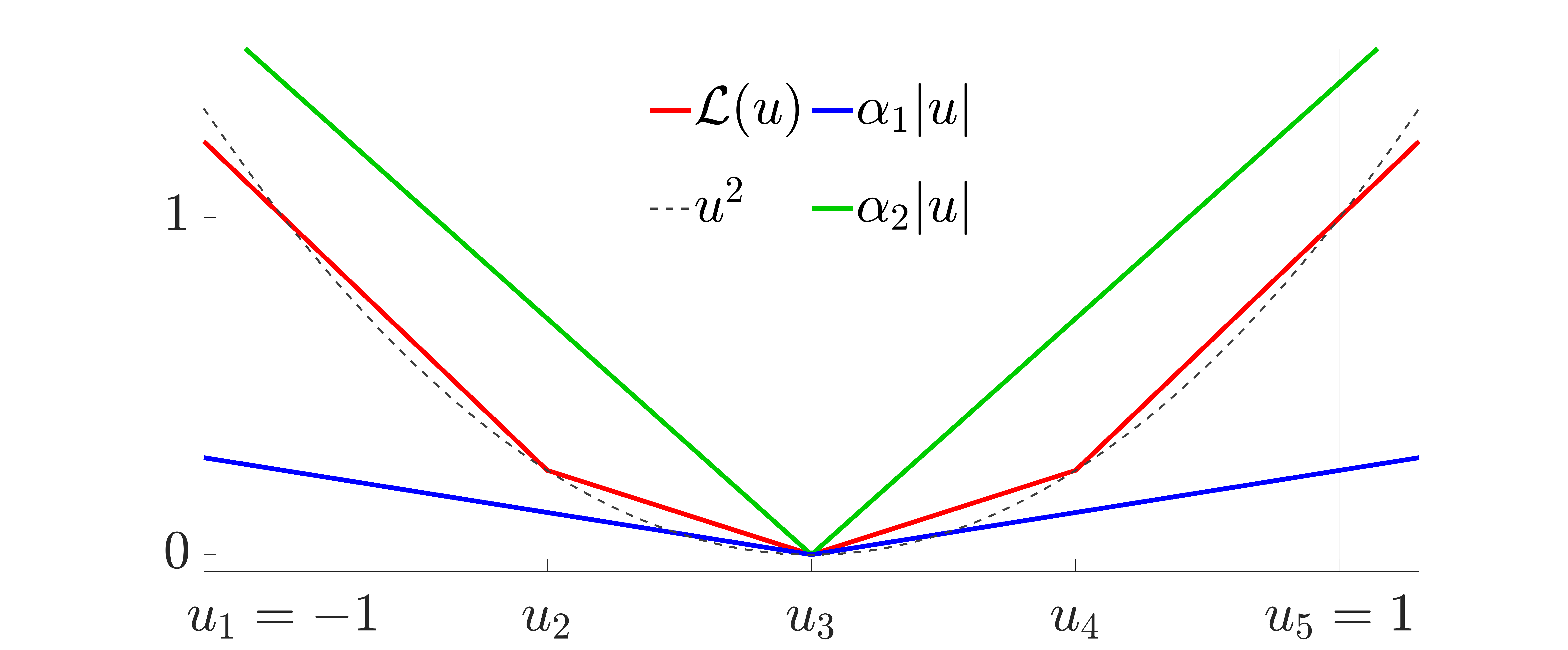}
		\caption{Upper and lower barriers for $\mathcal L(u)$ in terms of two multiples of $|u|$.}\label{fig:FunctCompare}
	\end{SCfigure}
	
	\item[2.] Since on the set $\mathcal U$ the functions $\mathcal P$ and $\mathcal L$ coincide, and since the minimum of $\mathcal P$ is reached  at some point in $\mathcal U$, then also $\mathcal L$ attains its minimum at the same point in $\mathcal U$. This, together with the coercivity in large time, will ensure the existence and uniqueness of a minimizer $\ptml$ for $J_{ml}$.  	
\end{itemize}

Finally, as we will see in the next section, the choice of the set $\mathcal U$ will determine the structural properties of the multilevel control. In particular:
\begin{itemize}
	\item[1.] The number of points in $\mathcal U$ will determine the maximum number of levels that the control may have, that is, how many constant values it may take, which can be at most $M$. 
	\item[2.] The distribution of the points in $\mathcal U$ will determine the different values $\{\rho_k\}_{k=1}^M$ of the multilevel control (see Definition \ref{def:multilevel}). As a matter of fact, as we will see in the proof of our main results stated below, these values $\rho_k$ will correspond to the slopes of the different linear branches of the penalization $\mathcal L$ which, of course, depend on how one chooses the interpolation points in $\mathcal U$. 
\end{itemize}

We stress that, as it will be clear from our proofs, the two observations above may be understood also back to front. We can first decide the number and amplitude of levels we want in our control. Then, we choose accordingly the set $\mathcal U$ and the strictly convex function $\mathcal P$, in such a way that the piece-wise linear penalization $\mathcal L$ interpolates $\mathcal P$ on $\mathcal U$ with the correspond slopes. 

We are now ready to present the main result of this paper, concerning the existence and uniqueness of a minimizer for the optimal control problems \eqref{eq:MultilevelFunct} and \eqref{eq:MultilevelFunctFabre}, the multilevel nature of the associated controls, and their staircase structure. In particular, we have the following theorems, whose proofs are given in the next section. 

\begin{theorem}\label{thm:MLtheorem}
Assume that $A\in\RR^{N\times N}$ satisfies either \eqref{eq:Acons} or \eqref{eq:Adiss} and let $B\in\RR^N$ be such that the pair $(A,B)$ fulfills the Kalman rank condition \eqref{eq:Kalman}. Let $\mathcal L:\RR\to\RR$ be constructed as in \eqref{eq:lambdaK}-\eqref{eq:Lml}. Then, there exists a positive time $T_\ast = T_\ast(x_0,A,\mathcal L)>0$ such that, for all $T\geq T_\ast$, the minimization problem \eqref{eq:MultilevelFunct} admits a unique solution $\ptml\in\RR^N$. Moreover, $\ptml$ uniquely determines a multilevel control $u_{ml}^\ast$ in the form \eqref{eq:uExpl} and satisfying the staircase property \eqref{eq:staircaseProp} such that, for any initial datum $x_0\in\RR^N$, the corresponding solution $x$ to \eqref{eq:primalSystem} fulfills $x(T)=0$.
\end{theorem}

\begin{theorem}\label{thm:MLtheorem2}
Assume that $A\in\RR^{N\times N}$ and $B\in\RR^N$ fulfill the Kalman rank condition \eqref{eq:Kalman}. Let $\mathcal L:\RR\to\RR$ be constructed as in \eqref{eq:lambdaK}-\eqref{eq:Lml}. Then, for all $T>0$, the minimization problem \eqref{eq:MultilevelFunct} admits a unique solution $\ptmld\in\RR^N$. Moreover, $\ptmld$ uniquely determines a multilevel control $\tilde{u}_{ml}^\ast$ in the form \eqref{eq:uExpl} and satisfying the staircase property \eqref{eq:staircaseProp} such that, for any initial datum $x_0\in\RR^N$, the corresponding solution $x$ to \eqref{eq:primalSystem} fulfills $x(T)=0$.
\end{theorem}

\section{Proof of the main result}\label{sec:proof} 

We prove here the main results of this paper, that is, Theorems \ref{thm:MLtheorem} and \ref{thm:MLtheorem2}. As we will see, in both cases, the existence and uniqueness of minimizers for the optimal control problems \eqref{eq:MultilevelFunct} and \eqref{eq:MultilevelFunctFabre} is a consequence of the direct method of calculus of variations. As for the multilevel nature of the corresponding controls, this will follow from the Euler-Lagrange equation associated with \eqref{eq:MultilevelFunct} and \eqref{eq:MultilevelFunctFabre}. Finally, the staircase structure will arise from the continuity properties of the solution to the adjoint equation \eqref{eq:adjoint}.

\begin{proof}[Proof of Theorem \ref{thm:MLtheorem}]

We are going to prove the result in two steps.

\subsubsection*{\textbf{Step 1: existence and uniqueness of a minimizer for \eqref{eq:MultilevelFunct}}}
The existence of a minimizer $\ptml$, solution to problem \eqref{eq:MultilevelFunct}, is a consequence of the direct method of calculus of variations. 

Observe that the functional $J_{ml}$ is strictly convex and continuous. Therefore, to ensure that it admits a minimum we only need to show that $J_{ml}$ is coercive, i.e.
\begin{align}\label{eq:coercivity}
	\lim_{|p_T|\to +\infty} J_{ml}(p_T) = +\infty.
\end{align}

At this regard, let us notice that, since we are assuming that the pair $(A,B)$ satisfies the Kalman rank condition \eqref{eq:Kalman} ensuring controllability, then we know that the following unique continuation property holds:
\begin{align*}
	B^\top p(t) = 0 \;\mbox{ for all }\; t\in [0,T]\quad\longrightarrow\quad p_T = 0.
\end{align*}
which is equivalent to the observability inequality 
\begin{align}\label{eq:observabilityL1}
	\int_0^T |B^\top p(t)|\,dt \geq \mathcal C_T |p_T|.
\end{align}

On the other hand, using \eqref{eq:compare}, there exist two positive constants $0<\alpha_1\leq\alpha_2<+\infty$, possibly dependent on $\mathcal L$, such that 
\begin{align}\label{eq:integralEst}
	\alpha_1\int_0^T |B^\top p(t)|\,dt \leq \int_0^T \mathcal L(B^\top p(t))\,dt \leq \alpha_2\int_0^T |B^\top p(t)|\,dt.
\end{align}

Combining \eqref{eq:observabilityL1} and \eqref{eq:integralEst}, we then obtain a new observability inequality for the solution of \eqref{eq:adjoint}, in the form
\begin{align}\label{eq:observabilityMl}
	\int_0^T \mathcal L(B^\top p(t))\,dt \geq \mathcal C_{T,ml} |p_T|,
\end{align}
with $\mathcal C_{T,ml} = \alpha_1(\mathcal L)\mathcal C_T$. Now, by definition of $J_{ml}$, and using \eqref{eq:observabilityMl}, we get that
\begin{align*}
	J_{ml}(p_T) \geq \mathcal C_{T,ml}|p_T| - |\langle x_0,p(0)\rangle_{\RR^N}|.
\end{align*}
On the other hand, by means of the Cauchy-Schwarz inequality we can estimate
\begin{align*}
	|\langle x_0,p(0)\rangle_{\RR^N}| = |\langle e^{TA}x_0,p_T\rangle_{\RR^N}| \leq |e^{TA}x_0|\,|p_T|.
\end{align*}
This yields,
\begin{align*}
	J_{ml}(p_T) \geq \Big(\mathcal C_{T,ml}-e^{TA}|x_0|\Big)|p_T|.
\end{align*}
Therefore, if  
\begin{align}\label{eq:Tcond}
	\mathcal C_{T,ml}>e^{TA}|x_0|,
\end{align}
we have \eqref{eq:coercivity} and $J_{ml}$ admits a minimum. Hence, we only have to show that \eqref{eq:Tcond} holds. To this end, we have to distinguish two cases.
\begin{itemize}
	\item[1.] When $A$ satisfies \eqref{eq:Acons}, for all $y\in\RR^N$ we have $\langle y,Ay\rangle_{\RR^N} = 0$. It is then easy to see that the energy of the free solution associated with \eqref{eq:primalSystem} is conserved in time:
	\begin{align*}
		\frac{d}{dt}\left|e^{tA}x_0\right|^2 = 2\langle e^{tA}x_0,Ae^{tA}x_0\rangle_{\RR^N} = 0 \quad\longrightarrow\quad \left|e^{tA}x_0\right| = |x_0|\;\mbox{ for all }t\geq 0.
	\end{align*}
	In view of this, \eqref{eq:Tcond} simplifies into
	\begin{align*}
		\mathcal C_{T,ml}> |x_0|.
	\end{align*}
	On the other hand, we also know (see e.g. \cite{miller2005controllability}) that for conservative dynamics the controllability constant scales linearly with $T$:
	\begin{align*}
		\mathcal C_{T,ml} = \gamma T.
	\end{align*}
	Hence, \eqref{eq:Tcond} holds provided that $T>T_\ast$, with $T_\ast = |x_0|/\gamma$. 
	
	\item[2.] When $A$ satisfies \eqref{eq:Adiss}, the dynamics is dissipative and we have that 
	\begin{align*}
		\left|e^{tA}x_0\right| \to 0 \;\mbox{ as }t\to +\infty.
	\end{align*}
	On the other hand, we also know (see e.g. \cite{fernandez2000null}) that for dissipative dynamics the controllability constant grows exponentially for $T$ small:
	\begin{align*}
		\mathcal C_{T,ml} \sim\gamma \exp\left(\frac{1}{T}\right), \quad\mbox{ as }\quad T\to 0^+.
	\end{align*}
	Hence, also in this case, there exists a minimal time $T_\ast>0$ such that, if $T>T_\ast$, the inequality \eqref{eq:Tcond} holds.
\end{itemize}

Therefore, both for a conservative and a dissipative dynamics, if the time horizon is large enough we have that \eqref{eq:Tcond} holds and $J_{ml}$ admits a minimizer $\ptml$. 

As for the uniqueness of $\ptml$, it is enough to notice that, by construction, the penalization $\mathcal L$ has a unique minimum at one of the points in the set $\mathcal U$. Indeed, recall that we are assuming the function $\mathcal P$ which $\mathcal L$ interpolates to be strictly convex and with its minimum attained at some $u_k\in\mathcal U$. However, since on $\mathcal U$ the functions $\mathcal P$ and $\mathcal L$ coincide, this $u_k$ must also be the minimum of $\mathcal L$.   

\subsubsection*{\textbf{Step 2: multilevel structure of the controls}}

Once we know that there exists a unique minimizer $\ptml$ solution of \eqref{eq:MultilevelFunct}, we can employ the adjoint methodology to define an optimal control for \eqref{eq:primalSystem} and analyze its properties.
	
To this end, let us start by deriving the Euler-Lagrange equation associated with $J_{ml}$. Following a standard approach, this is obtained by computing
\begin{align}\label{eq:epsilonDer}
	\frac{d}{d\epsilon} J_{ml}(\ptml + \epsilon p_T)\,\Big|_{\epsilon = 0} = 0 \quad\mbox{ for all } p_T\in\RR^N.
\end{align}
	
Notice, however, that the function $\mathcal L$ is not differentiable at the points $\{u_k\}_{k=1}^L$. Hence, in order to compute the above derivative, we need to introduce the sub-differential of $\mathcal L$, which is defined for any $u\in [\varpi_1,\varpi_2]$ as
\begin{align*}
	\partial \mathcal L(u) = \Big\{c\in \RR\;\text{s.t.}\;\mathcal L(\eta) - \mathcal L(u) \geq c(\eta-u),\;\text{ for all }\; \eta\in [\varpi_1,\varpi_2]\Big\}. 
\end{align*}
	
In the case of a convex function as $\mathcal L$, one can readily show that the sub-differential at $u\in (\varpi_1,\varpi_2)$ is the nonempty interval $[\beta^-,\beta^+]$, where $\beta^-$ and $\beta^+$ are the one-sided derivatives
\begin{equation*}
	\beta^- = \displaystyle\lim_{\eta\to u^-} \frac{\mathcal L(\eta) - \mathcal L(u)}{\eta-u} \quad\mbox{ and }\quad \beta^+ = \displaystyle\lim_{\eta\to u^+} \frac{\mathcal L(\eta) - \mathcal L(u)}{\eta-u}. 
\end{equation*}
Moreover, the sub-differential at $u=\varpi_1$ and $u=\varpi_2$ is given by $(-\infty,\beta^+]$ and $[\beta^-,+\infty)$ respectively.
	
Finally, notice that, if $\mathcal L$ is differentiable at some $u\in (\varpi_1,\varpi_2)$, then the left and the right derivatives coincide, and thus, $\partial \mathcal L(u)$ is just the classical derivative.
	
Using this characterization of the sub-differential, we can obtain from \eqref{eq:epsilonDer} that $\ptml\in\RR^N$ is a minimizer of $J_{ml}$ if and only if 
\begin{align}\label{eq:EulerLagrange}
	0 \in \int_0^T \partial \mathcal L(B^\top \pml(t))B^\top p(t)\,dt + \langle x_0,p(0)\rangle_{\RR^N} \quad\mbox{ for all } p_T\in\RR^N,
\end{align}
where $p$ denotes the solution of \eqref{eq:adjoint} corresponding to the final datum $p_T$. Let us now multiply equation \eqref{eq:primalSystem} by $p$ and integrate by parts. In this way, we easily get 
\begin{align}\label{eq:ControlIDprel}
	0 &= \int_0^T \langle \dot{x}(t)-Ax(t)-Bu(t),p(t)\rangle_{\RR^N}\,dt \notag 
	\\
	&= \langle x(T),p_T\rangle_{\RR^N} - \langle x_0,p(0)\rangle_{\RR^N} - \int_0^T \langle x(t),\dot{p}(t) + A^\top p(t)\rangle_{\RR^N}\,dt - \int_0^T \langle Bu(t),p(t)\rangle_{\RR^N}\,dt
	\\
	&= \langle x(T),p_T\rangle_{\RR^N} - \langle x_0,p(0)\rangle_{\RR^N} - \int_0^T u(t)B^\top p(t)\,dt. \notag 
\end{align}
Hence, if take the optimal control $u_{ml}^\ast$ such that 
\begin{align}\label{eq:controlMl}
	u_{ml}^\ast(t)\in\partial\mathcal L(B^\top \pml(t)),
\end{align}
we get from \eqref{eq:ControlIDprel} that 
\begin{align*}
	0\in \langle x(T),p_T\rangle_{\RR^N} - \langle x_0,p(0)\rangle_{\RR^N} - \int_0^T \partial\mathcal L(B^\top \pml(t))B^\top p(t)\,dt
\end{align*}
or, equivalently, 
\begin{align*}
	\langle x(T),p_T\rangle_{\RR^N} \in \int_0^T \partial\mathcal L(B^\top \pml(t))B^\top p(t)\,dt + \langle x_0,p(0)\rangle_{\RR^N}.
\end{align*}
	
We then see that \eqref{eq:primalSystem} is null controllable a time $T$, i.e. $x(T)=0$, if and only if the Euler-Lagrange equation \eqref{eq:EulerLagrange} is satisfied. Therefore, the unique minimizer $\ptml$ of the functional $J_{ml}$ determines through \eqref{eq:controlMl} a unique null control $u^\ast_{ml}(t)$ for \eqref{eq:primalSystem}. 
	
Thus, in order to conclude our proof, we only have to show that $u_{ml}^\ast$ defined in \eqref{eq:controlMl} is a multilevel control satisfying the staircase property \eqref{eq:staircaseProp}. To this end, we shall exploit the properties of $\partial\mathcal L$. 
	
Using the above characterization of the sub-differential, we can compute $\partial\mathcal L(u)$ for all $u\in [\varpi_1,\varpi_2]$. In order to do that, let us define
\begin{align}\label{eq:sigma_k} 
	\sigma_k := \frac{d}{du}\lambda_k(u) = \frac{\mathcal{P}(u_{k+1}) - \mathcal{P}(u_k)}{u_{k+1} - u_k}\quad \mbox{ for all } k\in \{1,\ldots,M\},
\end{align} 
with $\lambda_k(u)$ given by \eqref{eq:lambdaK}. Using the definition of $\mathcal L$ in \eqref{eq:Lml}, we can then compute
\begin{align*}
	&\partial\mathcal L(\varpi_1) = (-\infty,\sigma_1], 
	\\[5pt]
	&\partial\mathcal L(\varpi_2) = [\sigma_M,+\infty), 
	\\[5pt]
	&\partial\mathcal L(u_k) = [\sigma_{k-1},\sigma_k], \quad\mbox{ for all }k\in \{2,\ldots,M\},
\end{align*}
and
\begin{align*}
	\partial\mathcal L(u) = \{\sigma_k\}, \quad \mbox{ for all } u\in (u_k,u_{k+1}) \mbox{ and all } k\in \{ 1,\ldots,M\}.
	\end{align*}
In view of the above computations, and using \eqref{eq:controlMl}, we obtain that
\begin{displaymath}
	\begin{array}{ll}
	u_{ml}^\ast = \sigma_k, & \mbox{ if } B^\top \pml\in (u_k,u_{k+1}) \mbox{ and for all } k\in \{1,\ldots,M\},
	\\[8pt]
	u_{ml}^\ast\in \left(-\infty,\sigma_1\right], & \mbox{ if } B^\top \pml = \varpi_1, 
	\\[8pt]
	u_{ml}^\ast\in \left[\sigma_{k-1},\sigma_k\right], & \mbox{ if } B^\top \pml = u_k \mbox{ and for all } k\in \{ 1,\ldots,M\},
	\\[8pt]
	u_{ml}^\ast\in \left[\sigma_M,+\infty\right), & \mbox{ if } B^\top \pml = \varpi_2,
	\end{array} 
\end{displaymath}
provided that the set 
\begin{align*}
	\mathcal I_{ml} :=\Big\{t\in (0,T)\;:\; B^\top \pml(t) = u_k \mbox{ for some } k \in\{1,\ldots,M\}\Big\}
\end{align*}
has zero Lebesgue measure. This, however, is always true since the optimal solution $\pml$ is analytic and, therefore, 
\begin{align*}
	\mathcal I_{ml} = \{t_k\}_{k=1}^K, \quad \NN\ni K<+\infty.
\end{align*}
Let us now partition the time interval $(0,T)$ as
\begin{align*}
	(0,T)= \bigcup_{k=1}^K I_k, 
\end{align*}
where $I_k:=(t_k,t_{k+1})$ with $t_0=0$, $t_{K+1}=T$, and
\begin{align*}
	B^\top \pml(t)\in (u_k,u_{k+1}), \quad\mbox{ if } t\in I_k.
\end{align*}
Then, the optimal control $u^\ast$ is given explicitly by
\begin{align}\label{eq:controlMultilevel}
	u^\ast_{ml}(t) = \sum_{k=1}^K \sigma_k\chi_{I_k}(t)
\end{align}
In other words, the optimal control $u^\ast$ is in the form \eqref{eq:uExpl} with 
\begin{align}\label{eq:rho_k}
	s_k:=\sigma_k.
\end{align}
	
Notice that these $\sigma_k$ are nothing more than the slopes of the different linear branches of the penalization function $\mathcal L$ between the interpolation points. Hence, by defining $\mathcal L$ we also automatically define the constant values that the multilevel control $u^\ast_{ml}$ may assume.
	
As for the staircase property \eqref{eq:staircaseProp}, this is simply a consequence of the regularity of $B^\top \pml$. Indeed, since $B^\top \pml$ is a continuous function, if $B^\top \pml(t)\in I_K$ for some $t\in (0,T)$ then, for all $\varepsilon>0$, $B^\top \pml(t+\varepsilon)$ can only belong to $I_{k-1}$, $I_k$ or $I_{k+1}$. Hence, the multilevel control can only jump from a value in $\mathcal R$ to the immediate precedent or successive one, thus being in staircase form. Our proof is then concluded. 
\end{proof}

\begin{remark}
\em{ 
As we have seen in the proof of Theorem \ref{thm:MLtheorem}, in order to obtain the coercivity of $J_{ml}$, we need a large enough time horizon $T\geq T_\ast$ so that \eqref{eq:Tcond} holds. Moreover, this threshold $T_\ast$ depends on then initial datum $x_0$ and on the dynamics $A$. 

In fact, according to Theorem \ref{thm:MLtheorem}, multilevel controls are characterized by \eqref{eq:controlMl} and are piece-wise constant functions, whose values $\{\sigma_k\}_{k=1}^M$ are given by the slopes of the different piece-wise linear branches of the penalization $\mathcal L$. Hence, roughly speaking, this penalization $\mathcal L$ dictates the maximal intensity of the multilevel control. Then, once $\mathcal L$ is fixed, this yields to some constraints on the control, and it is therefore natural to expect the appearance of a minimal controllability time. This is in analogy with the known results for constrained controllability of PDEs (see, for instance, \cite{antil2019controllability,biccari2020controllability,loheac2017minimal,pighin2018controllability,pighin2019controllability}). 

At this regard, we shall remark that this minimal controllability time $T_\ast$ may be \textit{modulated} by introducing some simple modification in the functional $J_{ml}$. For instance, one may consider the following optimal control problem:
\begin{equation}\label{eq:MultilevelFunctModified}
	\begin{array}{l}
		\displaystyle p_{T,\beta,ml}^\ast = \min_{\underset{p\text{ solves }\eqref{eq:adjoint}}{p_T\in\RR^N}} J_{\beta,ml}(p_T) 
		\\[20pt]
		\displaystyle J_{\beta,ml}(p_T) = \beta\int_0^T \mathcal L(B^\top p(t))\,dt + \langle x_0,p(0)\rangle_{\RR^N},
	\end{array} 
\end{equation}
with $\RR\ni\beta>1$. Following the proof of Theorem \ref{thm:MLtheorem}, this would yield to multilevel controls in the form (see \eqref{eq:controlMultilevel})
\begin{align*}
	u^\ast_{\beta,ml}(t) = \sum_{k=1}^K \beta\sigma_k\chi_{I_k}(t),
\end{align*}
whose intensity is now amplified by a factor of $\beta$. Moreover, from \eqref{eq:Tcond} we obtain that the minimal controllability time would be defined through the condition  
\begin{align*}
	\mathcal C_{T,ml}>\frac{\mathcal C(x_0,A,T)}{\beta}
\end{align*}
and, choosing $\beta$ large, one may expect $T_\ast$ to become smaller.
}
\end{remark}

Before proving our second main result Theorem \ref{thm:MLtheorem2}, let us present a concrete example of a dynamical system in which the optimal control problem \eqref{eq:MultilevelFunct} may be unsuccessful in providing a multilevel control, not even considering a large time horizon. To this end, let us consider the linear scalar ODE 
\begin{align}\label{eq:example}
	\begin{cases}
		x'(t) = x(t) + u(t), \quad t\in(0,T)
		\\
		x(0) = x_0\in\RR,
	\end{cases}
\end{align}
and the associated adjoint dynamics
\begin{align*}
	\begin{cases}
		-p'(t) = p(t), \quad t\in(0,T)
		\\
		p(T) = p_T\in\RR.
	\end{cases}
\end{align*}

Suppose that we want to use the functional $J_{ml}$ to design a bang-bang control (i.e., a multilevel control with two levels) steering the dynamics \eqref{eq:example} from any $x_0\in\RR$ to zero at time $T$.

Following the proof of Theorem \ref{thm:MLtheorem}, in order to do that we need the coercivity of $J_{ml}$, which requires to have the observability inequality
\begin{align}\label{eq:observabilityEx}
	\int_0^T |p(t)|\,dt \geq \mathcal C_T |p_T|,
\end{align}
with $\mathcal C_T$ large enough so that \eqref{eq:Tcond} holds. In particular, in this specific case, we need
\begin{align}\label{eq:TcondEx}
	\mathcal C_T > e^T|x_0|.
\end{align}
Nevertheless, this is possible only for small-enough initial data. Indeed, we can easily compute
\begin{align*}
	\int_0^T |p(t)|\,dt = |p_T|\int_0^T e^{T-t}\,dt = |p_T|\left(e^T-1\right)
\end{align*}
and, plugging this into \eqref{eq:observabilityEx}, we obtain that
\begin{align*}
	|p_T|\left(e^T-1\right) \geq \mathcal C_T |p_T| \quad\longrightarrow\quad \mathcal C_T\leq e^T-1.
\end{align*}
Hence, \eqref{eq:TcondEx} can hold only if 
\begin{align*}
	|x_0| < 1-e^{-T}<1.
\end{align*}

This shows that, no matter the time horizon $T$, if we want to ensure the coercivity of $J_{ml}$, we need to impose some restrictions on the size of the initial datum $x_0$. Otherwise, we have no guarantee that the optimal control process \eqref{eq:MultilevelFunct} will be successful in providing a multilevel control for \eqref{eq:example}. 

This observation motivates the introduction of the alternative optimal control problem \eqref{eq:MultilevelFunctFabre}, which allows obtaining multilevel controls in any positive time $T>0$ and for any dynamics satisfying the Kalman rank condition.

\begin{proof}[Proof of Theorem \ref{thm:MLtheorem2}] 
	
Also in this case, we are going to prove the result in two steps.
	
\subsubsection*{\textbf{Step 1: existence and uniqueness of a minimizer for \eqref{eq:MultilevelFunctFabre}}}
The existence of a minimizer $\ptmld$, solution to problem \eqref{eq:MultilevelFunctFabre}, is once again a consequence of the direct method of calculus of variations. 

Observe that the functional $\mathcal J_{ml}$ is strictly convex and continuous. Therefore, to ensure that it admits a minimum we only need to show its coercivity.

As in the proof of Theorem \ref{thm:MLtheorem} before, this will follow from the observability inequality \eqref{eq:observabilityMl} which, this time, yields that 
\begin{align*}
	\mathcal J_{ml}(p_T) \geq \frac{C_{T,ml}^2}{2}|p_T|^2 -e^{TA}|x_0||p_T|.
\end{align*}

Since the first term in the above inequality is quadratic in $p_T$, while the second one is only linear, we then have that, for all $T>0$, 
\begin{align*}
	\lim_{|p_T|\to +\infty} \mathcal J_{ml}(p_T) = +\infty.
\end{align*}

Therefore, $\mathcal J_{ml}$ is coercive and admits a minimizer $\ptmld$. Moreover, this minimizer is unique since the functional $\mathcal J_{ml}$ is clearly strictly convex.   
	
\subsubsection*{\textbf{Step 2: multilevel structure of the controls}}
	
Once we know that there exists a unique minimizer $\ptmld$ solution of \eqref{eq:MultilevelFunctFabre}, we can employ the adjoint methodology to define an optimal control for \eqref{eq:primalSystem} and analyze its properties.
	
Following the proof of Theorem \ref{thm:MLtheorem}, we then need to compute the Euler-Lagrange equation associated with \eqref{eq:MultilevelFunctFabre}, which reads as 
	
\begin{align}\label{eq:EulerLagrange2}
	0 \in \Lambda_{T,ml}\int_0^T \partial \mathcal L(B^\top \pmld(t))B^\top p(t)\,dt + \langle x_0,p(0)\rangle_{\RR^N} \quad\mbox{ for all } p_T\in\RR^N,
\end{align}
where $p$ denotes the solution of \eqref{eq:adjoint} corresponding to the final datum $p_T$ and $\Lambda_{T,ml}$ is given in \eqref{eq:controlMlFabre}. From here, arguing as before, we obtain that the optimal control is characterized as
\begin{align}\label{eq:controlFabre}
	\tilde{u}_{ml}^\ast(t)\in\Lambda_{T,ml}\partial\mathcal L(B^\top \pmld(t)).
\end{align}

The remaining of the proof is analogous to the one of Theorem \ref{thm:MLtheorem} and we leave the details to the reader.
\end{proof}

\begin{remark}\label{rem:2controls}
\em{
To conclude this section, let us remark that Theorems \ref{thm:MLtheorem} and \ref{thm:MLtheorem2} are stated for a linear system \eqref{eq:primalSystem} with one scalar control. Nevertheless, our results may be extended to the case of multiple controls in \eqref{eq:primalSystem}, that is,
\begin{align*}
	\begin{cases}
		\displaystyle x'(t) = Ax(t) + \sum_{k=1}^K B_k u_k(t), & t\in(0,T)
		\\
		x(0) = x_0.
	\end{cases}
\end{align*}

As a matter of fact, it would be enough to consider some small modification in the functionals $J_{ml}$ and $\mathcal J_{ml}$ as follows:
\begin{equation}\label{eq:functionalsNew}
	\begin{array}{l}
		\displaystyle\widetilde{J}_{ml}(p_T) = \int_0^T \sum_{k=1}^K\mathcal L_k(B_k^\top p(t))\,dt + \langle x_0,p(0)\rangle_{\RR^N}
    	\\[20pt]
		\displaystyle\widetilde{\mathcal J}_{ml}(p_T) = \frac 12\left(\int_0^T \sum_{k=1}^K\mathcal L_k(B_k^\top p(t))\,dt\right)^2 + \langle x_0,p(0)\rangle_{\RR^N},
	\end{array} 
\end{equation}
where, $\{\mathcal L_k\}_{k=1}^K$ is a family of piece-wise linear penalizations all built according to \eqref{eq:lambdaK}-\eqref{eq:Lml}. 

Our proofs of Theorems \ref{thm:MLtheorem} and \ref{thm:MLtheorem2} can then be easily adapted to deal with these new functionals. We leave the details to the reader. Moreover, in Section \ref{sec:numerics}, we will present some numerical evidence of the efficacy of our control strategy also in this case.
}
\end{remark}

\section{Structural properties of the multilevel control problem}\label{sec:structural}

It is well-known that the adjoint formulation that we presented in Sections \ref{sec:problem} and \ref{sec:proof} is particularly suited to analyze structural properties for controllability and optimal control problems. In this section, we discuss some of those properties in the context of multilevel control. 

In order to simplify our presentation, for the remaining of this section we will consider the particular case
\begin{align*}
	\mathcal P(u) = u^2.
\end{align*}

The case of a general $C^2$ and strictly convex function $\mathcal P$ is an easy extension of the results we are going to present, whose details are left to the reader.

\subsection{Convergence to $L^2$ controls} 

As we have shown in our main result Theorem \ref{thm:MLtheorem}, to design multilevel controls for the linear system \eqref{eq:primalSystem} it is enough to minimize the functionals $J_{ml}$ or $\mathcal J_{ml}$ defined in \eqref{eq:MultilevelFunct} and \eqref{eq:MultilevelFunctFabre}, upon a suitable selection of the penalization function $\mathcal L$. This penalization is built as the piece-wise linear $(M+1)$-points interpolation of some given strictly convex function $\mathcal P$, to which it converges as $M\to +\infty$ (see Lemma \ref{lem:convergence}). In particular, since we are assuming that $\mathcal P(u) = u^2$, we have from \eqref{eq:errorInterpGlobal} that the interpolation error can be estimated by 
\begin{align*}
	e_{max}\leq h^2\to 0, \quad\mbox{ as } M\to +\infty.
\end{align*}

It is therefore very natural to analyze what happens to the multilevel control when increasing the number of interpolation points up to infinity. As a matter of fact, as one may expect, when $M\to +\infty$ multilevel controls converge to $L^2$ ones obtained through the minimization of the corresponding quadratic functional. In particular, we have the following result.

\begin{theorem}\label{thm:convergence}
Let $J_{ml}$ be the cost functional defined in \eqref{eq:MultilevelFunct} with $\mathcal L$ given by \eqref{eq:lambdaK}-\eqref{eq:Lml}. Fix $T>0$ large enough fulfilling \eqref{eq:Tcond}, such that $J_{ml}$ admits a unique minimizer $\ptml\in\RR^N$, and $\uml\in L^\infty(0,T;\mathcal R)$ be the associated multilevel control given by \eqref{eq:controlMl}. Moreover, let $\pt\in\RR^N$ be the unique solution of the minimization problem
\begin{equation*}\label{eq:L2funct}
	\begin{array}{l}
		\displaystyle \pt = \min_{\underset{p\text{ solves }\eqref{eq:adjoint}}{p_T\in\RR^N}} J_2(p_T)
		\\[20pt]
		\displaystyle J_2(p_T) := \int_0^T |B^\top p(t)|^2\,dt + \langle x_0,p(0)\rangle_{\RR^N}.
	\end{array} 
\end{equation*}
and let $\uu = B^\top p_2^\ast\in L^2(0,T;\RR)$ be the corresponding control, with $p_2^\ast$ the unique solution of \eqref{eq:adjoint} corresponding to the initial datum $\pt$. Then, the following holds:
\begin{itemize}
	\item[1.] $J_{ml}\to J_2$ as $M\to +\infty$ a.e. in $\RR$.
	\item[2.] $\uml\to\uu$ as $M\to +\infty$ a.e. in $(0,T)$.
\end{itemize}
\end{theorem}

\begin{proof}
The first result is a direct consequence of Lemma \ref{lem:convergence}. Indeed, since we know that $\mathcal L(u)\to u^2$ a.e. in $[-\varpi,\varpi]$ as $M\to +\infty$, we immediately have that
\begin{align*}
	\left|\int_0^T \mathcal L(B^\top p(t))\,dt - \int_0^T \mathcal |B^\top p(t)|^2\,dt\,\right| \leq \int_0^T \Big|\mathcal L(B^\top p(t)) - |B^\top p(t)|^2\,\Big|\,dt \to 0,
\end{align*}
thus yielding 
\begin{align*}
	\int_0^T \mathcal L(B^\top p(t))\,dt \to \int_0^T \mathcal |B^\top p(t)|^2\,dt \quad \mbox{ a.e. in }\RR\mbox{ as }M\to+\infty.
\end{align*}

As for the convergence of the optimal controls, since $J_{ml}\to J_2$ a.e. in $\RR$ as $M\to +\infty$, from the uniqueness of the minimizers we also have that $\ptml\to\pt$. In particular, this convergence transfers to the optimal solutions of the corresponding adjoint equations:
\begin{align*}
	p_{ml}^\ast(t) = e^{(T-t)A^\top}\ptml \to e^{(T-t)A^\top}\pt = p_2^\ast(t) \quad \mbox{ a.e. in } (0,T) \mbox{ as } M\to +\infty.
\end{align*} 
Hence, clearly,
\begin{align*}
	B^\top p_{ml}^\ast(t) \to B^\top p_2^\ast(t) \quad \mbox{ a.e. in } (0,T) \mbox{ as } M\to +\infty.
\end{align*} 
Finally, we have from \cite[Theorem 4.2]{attouch1993convergence} and Lemma \ref{lem:convergence} that 
\begin{align*}
	\lim_{M\to +\infty} \partial\mathcal L(B^\top p_{ml}^\ast(t)) = \partial ((B^\top p_2^\ast(t))^2) = B^\top p_2^\ast(t)
\end{align*}
and that, for all $v\in\partial ((B^\top p_2^\ast(t))^2)$, there exists a $u_M\in\partial\mathcal L(B^\top p_{ml}^\ast(t))$ such that 
\begin{align*}
	\lim_{M\to +\infty} u_M = v.
\end{align*}
The result then follows by the uniqueness of the optimal controls.
\end{proof}

Analogously, for controls obtained via the optimal control problem \eqref{eq:MultilevelFunctFabre}, we have the following result whose proof is left to the reader. 

\begin{theorem}\label{thm:convergence2}
Let $\mathcal J_{ml}$ be the cost functional defined in \eqref{eq:MultilevelFunctFabre} with $\mathcal L$ given by \eqref{eq:lambdaK}-\eqref{eq:Lml}. For any $T>0$, let $\ptmld\in\RR^N$ be the unique minimizer of $\mathcal J_{ml}$ and $\tilde{u}_{ml}^\ast\in L^\infty(0,T;\mathcal R)$ be the associated multilevel control given by \eqref{eq:controlFabre}. Moreover, let $\tilde{p}_{T,2}^\ast\in\RR^N$ be the unique solution of the minimization problem
\begin{equation*}\label{eq:L2funct}
	\begin{array}{l}
		\displaystyle \tilde{p}_{T,2}^\ast = \min_{\underset{p\text{ solves }\eqref{eq:adjoint}}{p_T\in\RR^N}} \mathcal J_2(p_T)
		\\[20pt]
		\displaystyle \mathcal J_2(p_T) := \frac 12\left(\int_0^T |B^\top p(t)|^2\,dt\right)^2 + \langle x_0,p(0)\rangle_{\RR^N}.
	\end{array} 
\end{equation*}
and let $\tilde{u}_2^\ast = \Lambda_{T,2}B^\top \tilde{p}_2^\ast\in L^2(0,T;\RR)$ be the corresponding control, with $\tilde{p}_2^\ast$ the unique solution of \eqref{eq:adjoint} corresponding to the initial datum $\tilde{p}_{T,2}^\ast$ and
\begin{align*}
	\Lambda_{T,2}=:\int_0^T |B^\top \tilde{p}_2^\ast(t)|^2\,dt.
\end{align*}
Then, the following holds:
\begin{itemize}
	\item[1.] $\mathcal J_{ml}\to \mathcal J_2$ as $M\to +\infty$ a.e. in $\RR$.
	\item[2.] $\tilde{u}_{ml}^\ast\to\tilde{u}_2^\ast$ as $M\to +\infty$ a.e. in $(0,T)$.
\end{itemize}
\end{theorem}

\subsection{Fenchel-Rockafellar duality} One of the funding pillars of optimal control theory is that a convex optimization problem can be solved by applying duality in the sense of Fenchel and Rockafellar (see \cite{ekeland1999convex}). In this section, we are going to show that this duality approach is applicable also in our context. Actually, as we will see, the optimal control problem \eqref{eq:MultilevelFunct} that we are considering can be obtained as the Fenchel-Rockafellar dual of another optimal control problem still giving multilevel controls. 

\begin{theorem}\label{thm:Duality}
Let 
\begin{align*}
	\mathcal L^\star(v) = \sup_{u\in\RR} \Big(uv - \mathcal L(u)\Big)
\end{align*}
denote the convex conjugate of the penalization $\mathcal L$ defined in \eqref{eq:lambdaK}-\eqref{eq:Lml}, and consider the optimal control problem
\begin{subequations}
\begin{align}
	&\displaystyle v^\ast = \min_{v\in L^\infty(0,T;\RR)} \int_0^T \mathcal L^\star(v(t))\,dt\label{eq:FenchelFunct}
	\\[10pt]
	&\mbox{subject to } \begin{cases} x'(t) = Ax(t) + Bv(t), & t\in (0,T) \\ x(0) = x_0,\;x(T) = 0 \end{cases}.\label{eq:FenchelConstraint}
\end{align}
\end{subequations}
Then, it holds the following:
\begin{itemize}
	\item[1.] The functional $J_{ml}$ defined in \eqref{eq:MultilevelFunct} is obtained as the Fenchel-Rockafellar dual of \eqref{eq:FenchelFunct}.
	\item[2.] The optimal control $v^\ast$ obtained by \eqref{eq:FenchelFunct}-\eqref{eq:FenchelConstraint} coincides with the one obtained through the dual optimization process \eqref{eq:MultilevelFunct} and, in particular, has the multilevel structure.
\end{itemize}
\end{theorem}

\begin{proof} We split the proof into two steps. 
	
\subsubsection*{\textbf{Step 1. Fenchel-Rockafellar duality}} Let us start by proving the first point of the theorem. To this end, let us first observe that the solution to \eqref{eq:FenchelConstraint} can be written as $x=z+\xi$, with 
\begin{align}\label{eq:zEq}
	\begin{cases}
		z'(t) = Az(t) + Bv(t), & t\in(0,T)
		\\
		z(0) = 0
	\end{cases}
\end{align}
and
\begin{align}\label{eq:xiEq}
	\begin{cases}
		\xi'(t) = A\xi(t), & t\in(0,T)
		\\
		\xi(0) = x_0
	\end{cases}.
\end{align}
Then, problem \eqref{eq:FenchelFunct}-\eqref{eq:FenchelConstraint} can be rewritten as
\begin{subequations}
	\begin{align}
	&\displaystyle v^\ast = \min_{v\in L^\infty(0,T;\RR)} \int_0^T \mathcal L^\star(v(t))\,dt\label{eq:FenchelFunctZ}
	\\[10pt]
	&\mbox{subject to } \eqref{eq:zEq}\mbox{ with }z(T) = -\xi(T).\label{eq:FenchelConstraintZ}
	\end{align}
\end{subequations}
For $v\in L^\infty(0,T;\RR)$ and $p_T\in\RR^N$, define the following operators
\begin{align*}
	F_1(v):= \int_0^T \mathcal L^\star(v(t))\,dt, \quad F_2(p_T):= \begin{cases} 0, & \mbox{ if } p_T = -\xi(T) \\ +\infty & \mbox{ otherwise }\end{cases} \quad\mbox{ and }\quad Lv:=z(T),
\end{align*}
and notice that $F_2$ is a proper lower semi-continuous convex functional. Then, problem \eqref{eq:FenchelFunctZ}-\eqref{eq:FenchelConstraintZ} is equivalent to
\begin{align}\label{eq:FenchelFunct2}
	v^\ast = \min_{v\in L^\infty(0,T;\RR)} \Big(F_1(v) + F_2(Lv)\Big).
\end{align}

We can now apply the duality theory of Fenchel and Rockafellar (see \cite[Chapters VI and VII]{ekeland1999convex}), according to which
\begin{align*}
	\min_{v\in L^\infty(0,T;\RR)} \Big(F_1(v) + F_2(Lv)\Big) = \min_{p_T\in \RR^N} \Big(F_1^\star(L^\star p_T) + F_2^\star(-p_T)\Big),
\end{align*}
where $F^\star_1,F^\star_2$ are the conjugate of $F_1,F_2$, respectively, and $L^\star$ is the adjoint of $L$. 

It can be readily shown that the operator $L^\star$ is given by $L^\star p_T = B^\top p$, with $p$ the unique solution of the adjoint equation \eqref{eq:adjoint}. 

Let us now compute the convex conjugates $F^\star_1$ and $F^\star_2$. Taking into account that $\mathcal L$ and $\mathcal L^\star$ are both convex functions, and using \cite[Theorem 2]{rockafellar1968integrals}, we have that
\begin{align*}
	F_1^\star(L^\star p_T) = \int_0^T \mathcal L^{\star\star}(L^\star p_T)\,dt = \int_0^T \mathcal L(B^\top p(t))\,dt.
\end{align*} 
As for $F^\star_2$, we can easily see through the definition that 
\begin{align*}
	F_2^\star(p_T) = \sup_{q_T\in\RR^N} \langle p_T,q_T\rangle_{\RR^N} = -\langle p_T,\xi(T)\rangle_{\RR^N}.
\end{align*}
Hence, clearly, 
\begin{align*}
	F_2^\star(-p_T) = \langle p_T,\xi(T)\rangle_{\RR^N}.
\end{align*}
Moreover, using the fact that $\xi$ is the solution of \eqref{eq:xiEq}, we have 
\begin{align*}
	\langle p_T,\xi(T)\rangle_{\RR^N} = \langle p_T,e^{TA}x_0\rangle_{\RR^N} = \langle e^{-TA^\top}p_T,x_0\rangle_{\RR^N} = \langle x_0,p(0)\rangle_{\RR^N}.
\end{align*}
Putting everything together, we then have that 
\begin{align*}
	F_1^\star(L^\star p_T) + F_2^\star(-p_T) = \int_0^T \mathcal L(B^\top p(t))\,dt + \langle x_0,p(0)\rangle_{\RR^N}. 
\end{align*}
This shows that the functional $J_{ml}$ defined in \eqref{eq:MultilevelFunct} is the Fenchel-Rockafellar dual of \eqref{eq:FenchelFunct}.

\subsubsection*{\textbf{Step 2. Multilevel structure of $v^\ast$}} 

Let us now show that the optimal control $v^\ast$ obtained through the minimization process \eqref{eq:FenchelFunct}-\eqref{eq:FenchelConstraint} has the multilevel structure. 

As for Theorem \ref{thm:MLtheorem} above, the proof will be based on the Euler-Lagrange equation associated with \eqref{eq:FenchelFunct2}, which reads as
\begin{align}\label{eq:ELprel}
	0\in \int_0^T \partial\mathcal L^\star(v^\ast(t))w(t)\,dt + F_2(Lv) \quad\mbox{ for all } w\in L^\infty(0,T;\RR).
\end{align}
Let $p^\ast$ denote the solution of the adjoint equation
\begin{align*}
	\begin{cases}
		-p'(t) = A^\top p(t), & t\in (0,T)
		\\
		p(T) = -x^\ast(T)
	\end{cases},
\end{align*}
where $x^\ast$ is the solution of \eqref{eq:FenchelConstraint} with control $v^\ast$. Multiplying \eqref{eq:zEq} by $p^\ast$ and integrating over $(0,T)$, and taking into account the definition of $F_2$ and $Lv$, we can readily check that \eqref{eq:ELprel} is equivalent to
\begin{align*}
	0\in \int_0^T \Big(\partial\mathcal L^\star(v^\ast(t))-B^\top p^\ast(t)\Big)w(t)\,dt \quad\mbox{ for all } w\in L^\infty(0,T;\RR).
\end{align*}
Hence, the optimal control $v^\ast$ has to satisfy 
\begin{align}\label{eq:vAastCondition}
	B^\top p^\ast\in\partial\mathcal L^\star (v^\ast).
\end{align}
Finally, thanks to \eqref{eq:subdiffEq}, this is equivalent to 
\begin{align*}
	v^\ast\in \partial\mathcal L(B^\top p^\ast).
\end{align*}

We then see that $v^\ast$ coincides with the optimal control obtained through the dual optimization process \eqref{eq:MultilevelFunct} and, therefore, it has a multilevel structure. Our proof is then concluded.
\end{proof}

\begin{remark}
\em{
For completeness, we shall notice that the multi-level structure of $v^\ast$ could have also been inferred directly from \eqref{eq:vAastCondition}. Indeed, we know from Lemma \ref{lem:piecewiseConj} that $\mathcal L^\star(v^\ast)$ is still a piece-wise linear function and, therefore, $\partial \mathcal L^\star(v^\ast)$ is piece-wise constant. Then, the characterization \eqref{eq:vAastCondition} would eventually lead to the multilevel structure of $v^\ast$. For the sake of brevity, we leave the details to the reader.
}
\end{remark}

\subsection{Characterization of the solvable set}

In Sections \ref{sec:problem} and \ref{sec:proof}, we have discussed the need of a minimal controllability time $T_\ast$ to guarantee the coercivity of the functional $J_{ml}$ in \eqref{eq:MultilevelFunct} and, therefore, the existence of a multilevel control for any initial datum $x_0\in\RR^N$. 

Nevertheless, in certain practical situations (see \cite{oroya2021multilevel}), one may face with models in which the time horizon is predetermined by the specific scenario the system describes. In this case, it is clear that, to have the estimate \eqref{eq:Tcond} ensuring the coercivity of $J_{ml}$, one needs to assume that the initial datum $x_0$ is small enough. This motivates the introduction of the \textit{solvable set} for the multilevel control problem \eqref{eq:MultilevelFunct}, which is defined as follows. 

\begin{definition}
We define the \textit{solvable set} $\Sigma_{ml}$ as 
\begin{align}\label{eq:solvableSet}
	\Sigma_{ml}:= \Big\{x_0\in\RR^N\;:\; &\mbox{ for all } T>0 \mbox{ fixed there exists a multilevel control } u^\ast_{ml} \mbox{ obtained } 
	\\
	& \mbox{ through }\eqref{eq:MultilevelFunct} \mbox{ such that the corresponding solution } x \mbox{ to } \eqref{eq:primalSystem} \mbox{ satisfies } x(T) = 0\Big\}. \notag
\end{align}
\end{definition}

In this sub-section, we are going to show how the adjoint methodology proposed in this paper can allow for some characterization of the solvable set. In particular, we have the following result.
\begin{proposition}\label{prop:SolvableSet}
Fix $T>0$ such that \eqref{eq:MultilevelFunct} admits a unique minimizer $p_T^\ast\in\RR^N$. Let 
\begin{align*}
	u^\ast_{ml}(t) = \sum_{k=1}^K \sigma_k\chi_{I_k}(t)
\end{align*}
be the corresponding multilevel control, and define 
\begin{align}\label{eq:sigma_kBound}
	\bar{\sigma}:= \sup_{k=1,\ldots,M} |\sigma_k|,
\end{align}
Let $\Sigma_{ml}$ be the solvable set defined in \eqref{eq:solvableSet}. Then, for all $x_0\in\Sigma_{ml}$ the following estimate holds
\begin{align}\label{eq:x0necessary}
	\norm{x_0}{\RR^N} \leq \bar{\sigma} \norm{e^{-\tau A}B\,}{L^2(0,T;\RR^N)}. 
\end{align}
\end{proposition}

\begin{proof} 
First of all, note that since $\mathcal P\in C^2([-\varpi,\varpi])$ we clearly have that $\bar{\sigma} < +\infty$. Moreover, by means of the variation of constants formula, we can write the solution of \eqref{eq:primalSystem} as
\begin{align*}
	x(t) = e^{tA}x_0 + \int_0^t e^{(t-\tau)A}Bu(\tau)\,d\tau.
\end{align*}
Hence, the controllability condition $x(T)=0$ is equivalent to
\begin{align*}
	x_0 = -\int_0^T e^{-\tau A}Bu(\tau)\,d\tau.
\end{align*}

Taking into account the specific form of the multilevel control given in \eqref{eq:controlMultilevel}, we then have that every $x_0\in\Sigma_{ml}$ can be characterized as
\begin{align}\label{eq:x0expl}
	x_0 = -\sum_{k=1}^K \sigma_k \int_{t_k}^{t_{k+1}}e^{-\tau A}B\,d\tau. 
\end{align}
Therefore, we obtain from \eqref{eq:x0expl} and \eqref{eq:sigma_kBound} that
\begin{align*}
	\norm{x_0}{\RR^N}^2 &=\norm{\sum_{k=1}^K \sigma_k\int_{t_k}^{t_{k+1}}e^{-\tau A}B\,d\tau}{\RR^N}^2 \leq \bar{\sigma}^{\,2} \norm{\sum_{k=1}^K \int_{t_k}^{t_{k+1}}e^{-\tau A}B\,d\tau}{\RR^N}^2 \leq \bar{\sigma}^{\, 2} \sum_{k=1}^K \int_{t_k}^{t_{k+1}}\norm{e^{-\tau A}B\,}{\RR^N}^2\,d\tau 
	\\
	&= \bar{\sigma}^{\,2} \int_0^T\norm{e^{-\tau A}B\,}{\RR^N}^2\,d\tau = \bar{\sigma}^{\,2} \norm{e^{-\tau A}B\,}{L^2(0,T;\RR^N)}^2. 
\end{align*}
This, of course, gives immediately \eqref{eq:x0necessary}.
\end{proof} 

To conclude this section, let us remark that Proposition \ref{prop:SolvableSet} does not give a full description of the solvable set for the multilevel control problem. It just provides some necessary condition for an initial datum $x_0\in\RR^N$ to belong to $\Sigma_{ml}$. This necessary condition is expressed in terms of a bound on the norm of $x_0$ with respect to the dynamics, the time horizon for control and the control's intensity, telling us that the solvable set $\Sigma_{ml}$ is contained in some ball of $\RR^N$. Nevertheless, this does not exclude that some initial data may belong to that ball but not to the solvable set. 

A more precise characterization of $\Sigma_{ml}$ is a quite delicate issue which, although interesting, goes beyond the scope of the present paper and, therefore, will not be discussed further.

\section{Numerical simulations}\label{sec:numerics}

We present here some numerical simulations showing that the adjoint methodology described in Section \ref{sec:problem} indeed allows to compute multilevel controls. To this end, we consider a simple but very illustrative example: the control of an harmonic oscillator, that is the system
\begin{align}\label{eq:oscillator}
	\begin{cases}
		x_1'(t) = x_2(t), \;\;\; x_2'(t) = -x_1(t) + u(t), & t\in(0,T)
		\\
		x_1(0) = x_{1,0}, \;\;\;\; x_2(0) = x_{2,0}
	\end{cases}
\end{align}
This corresponds to system \eqref{eq:primalSystem} with the matrices $A$ and $B$ given by
\begin{align*}
	A = \left(\begin{matrix} 0 & 1 \\ -1 & 0 \end{matrix}\right)\in\RR^{2\times 2}\quad\quad\mbox{ and }\quad\quad B= \left(\begin{matrix} 0 \\ 1 \end{matrix}\right)\in\RR^2.
\end{align*}
Notice that \eqref{eq:oscillator} is null controllable at any time $T$ since $(A,B)$ satisfy the Kalman rank condition:
\begin{align*}
	\mbox{rank}\Big(B\,|AB\Big) = \mbox{rank}\left(\begin{matrix} 0 & 1 \\ 1 & 0 \end{matrix}\right) = 2.
\end{align*}

Moreover, the dynamics of \eqref{eq:oscillator} is conservative, as the matrix $A$ satisfies \eqref{eq:Acons}. Hence, according to our main results Theorems \ref{thm:MLtheorem} and \ref{thm:MLtheorem2}, we have two possibilities to obtain a multilevel control:
\begin{itemize}
	\item[1.] In a large time horizon, we can solve the optimal control problem \eqref{eq:MultilevelFunct}, and the multilevel control will be given by \eqref{eq:controlMl}.
	\item[2.] In a short time horizon, in which the functional $J_{ml}$ might not have a minimizer, we can instead solve the optimal control problem \eqref{eq:MultilevelFunctFabre}, and the multilevel control will be given by \eqref{eq:controlFabre}.
\end{itemize}

In what follows, we are going to present some numerical evidences of the above facts. In order to do that, we first fix a large time horizon $T=4$ and the initial datum $x_0 = (-1,0.5)^\top$, and employ a standard gradient descent methodology to compute the minimum $\ptml$ of the functional $J_{ml}$, from which we then obtain the multilevel control $u_{ml}^\ast$ through the characterization \eqref{eq:controlMl}. The penalization $\mathcal L$ in the functional $J_{ml}$ is constructed through \eqref{eq:Lml} with $\mathcal P(u) = u^2$, $u_1=-1$, $u_M = 1$ and $M=5$, thus producing a four-levels staircase control. 

In Figure \ref{fig:state}, we display the free and controlled dynamics of the linear system \eqref{eq:oscillator} under the action of this multilevel control $u_{ml}^\ast$. We can clearly see that, while the free states exhibit the expected oscillatory behavior, the introduction of the control allows to reach the zero state at time $T$.

\begin{figure}[h]
	\centering 
	\includegraphics[scale=0.4]{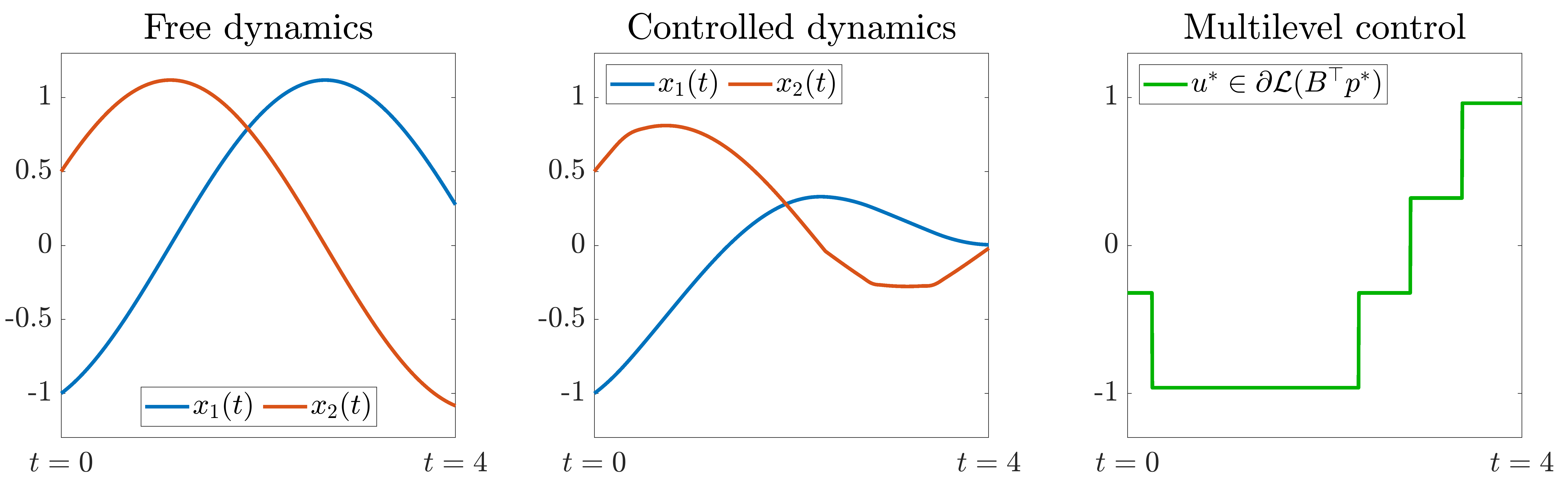}
	\caption{Free (left) and controlled (middle) dynamics of the linear system \eqref{eq:oscillator} under the action of the multilevel control $u_{ml}^\ast$ computed via the minimization of the cost functional $J_{ml}$ on the time horizon $T=4$.}\label{fig:state}
\end{figure}

This shows that our adjoint methodology is indeed successful in solving the multilevel control problem for \eqref{eq:oscillator}. 

Moreover, we show in Figure \ref{fig:controlConvergence} the behavior of the multilevel control for different increasing values of the parameter $M$. 

\begin{figure}[h]
	\centering 
	\includegraphics[scale=0.35]{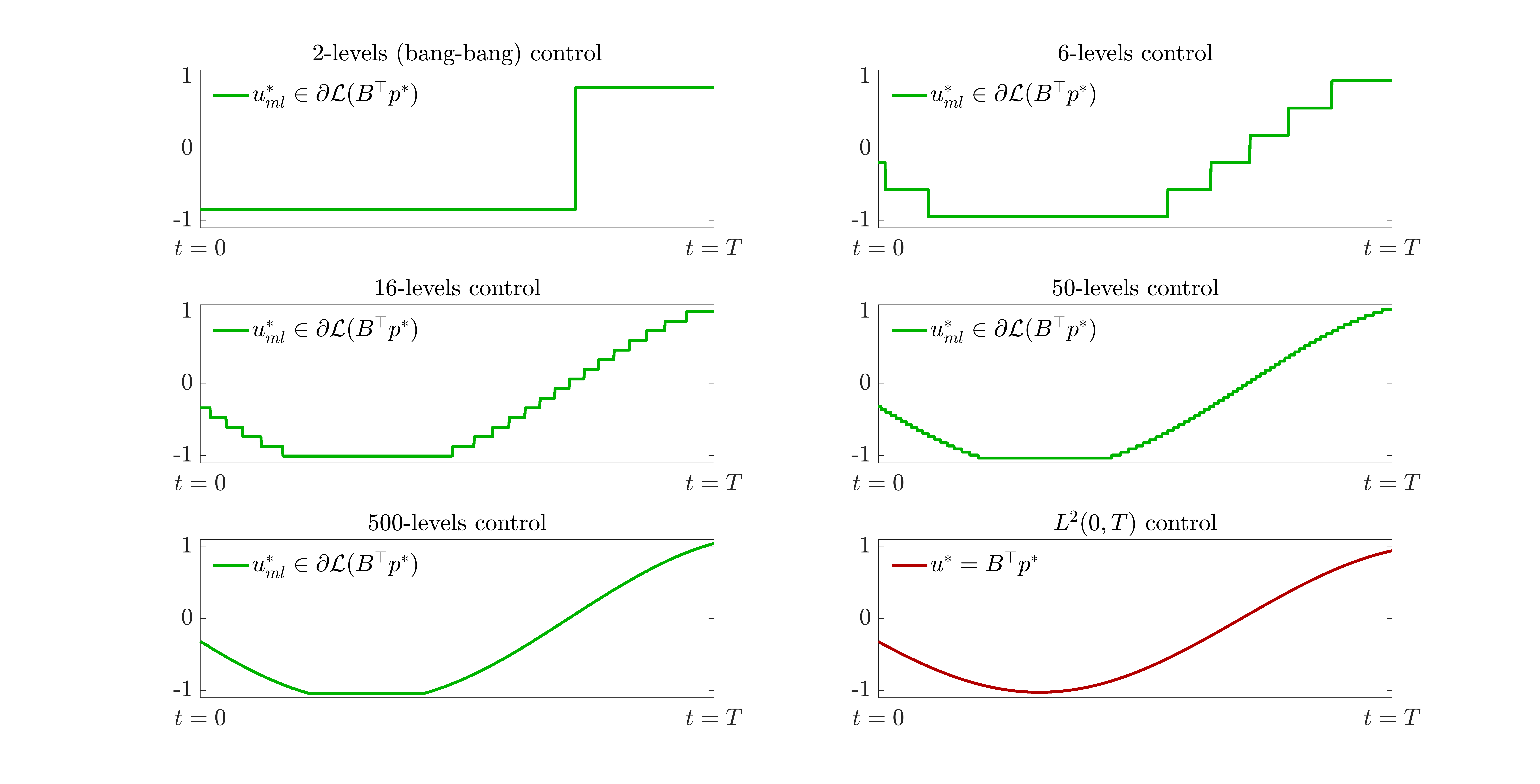}
	\caption{Optimal control $u^\ast_{ml}$ computed through the minimization of $J_{ml}$ with a penalization $\mathcal L$ built according to \eqref{eq:Lml} with increasing number of levels.}\label{fig:controlConvergence}
\end{figure}

According to Theorem \ref{thm:convergence}, when $M\to +\infty$ the multilevel control converges to the $L^2(0,T)$-control obtained through the minimization of the functional $J_2$. This behavior is indeed observed in Figure \ref{fig:controlConvergence}, where we clearly see that, as the number of levels increases, the control loses its multilevel nature until eventually converge to the $L^2(0,T)$-control.

Finally, we have seen in Remark \ref{rem:2controls} that our adjoint approach is still applicable when including several controls acting on the dynamics. To provide some numerical evidence of this fact, we have considered again system \eqref{eq:oscillator} with the same initial datum $x_0 = (-1,0.5)^\top$ and time horizon $T=4$, but this time with 
\begin{align*}
	Bu = B_1 u_1 + B_2 u_2, \quad B_1 = \left(\begin{matrix} 1 \\ 1 \end{matrix}\right), \quad B_2 = \left(\begin{matrix} 0 \\ 1 \end{matrix}\right).
\end{align*}

In Figure \ref{fig:state2controls}, we display the results of our numerical simulations, in which we have minimized the functional $\widetilde{J}_{ml}$ in \eqref{eq:functionalsNew} to obtain the multilevel controls $u_1$ and $u_2$. As we can see, our strategy is successful also in this situation.

\begin{figure}[h]
	\centering 
	\includegraphics[scale=0.38]{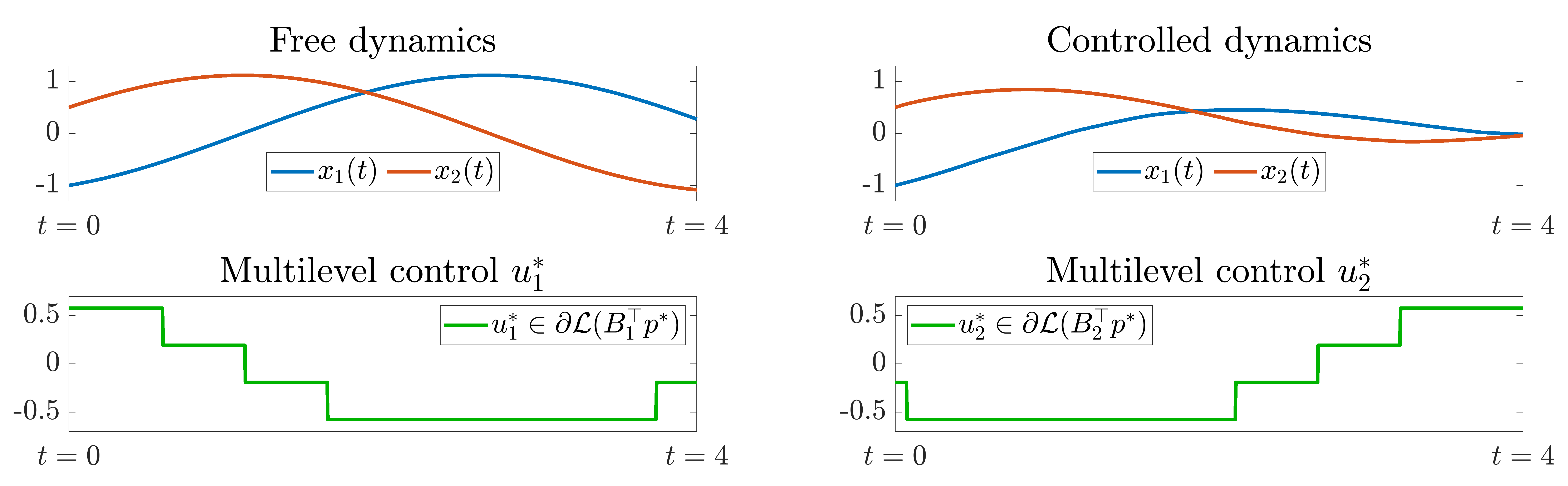}
	\caption{Free (left) and controlled (right) dynamics of the linear system \eqref{eq:oscillator} under the action of two multilevel controls $u_{ml,1}^\ast$ and $u_{ml,2}^\ast$ computed via the minimization of the cost functional $\widetilde{J}_{ml}$ on the time horizon $T=4$.}\label{fig:state2controls}
\end{figure}

Let us now consider a short time horizon, namely $T=0.5$. In this case, according to our theoretical results, we do not expect the optimal control process \eqref{eq:MultilevelFunct} to be successful in providing a multilevel control. This is indeed observed in Figure \ref{fig:stateTsmall}, where we are showing the free dynamics of \eqref{eq:oscillator} and the controlled one under the action of the multilevel control obtained by minimizing $J_{ml}$. We can clearly see that, despite the introduction of this control, the solution of \eqref{eq:oscillator} does not reach zero in time $T$.

\begin{figure}[h]
	\centering 
	\includegraphics[scale=0.38]{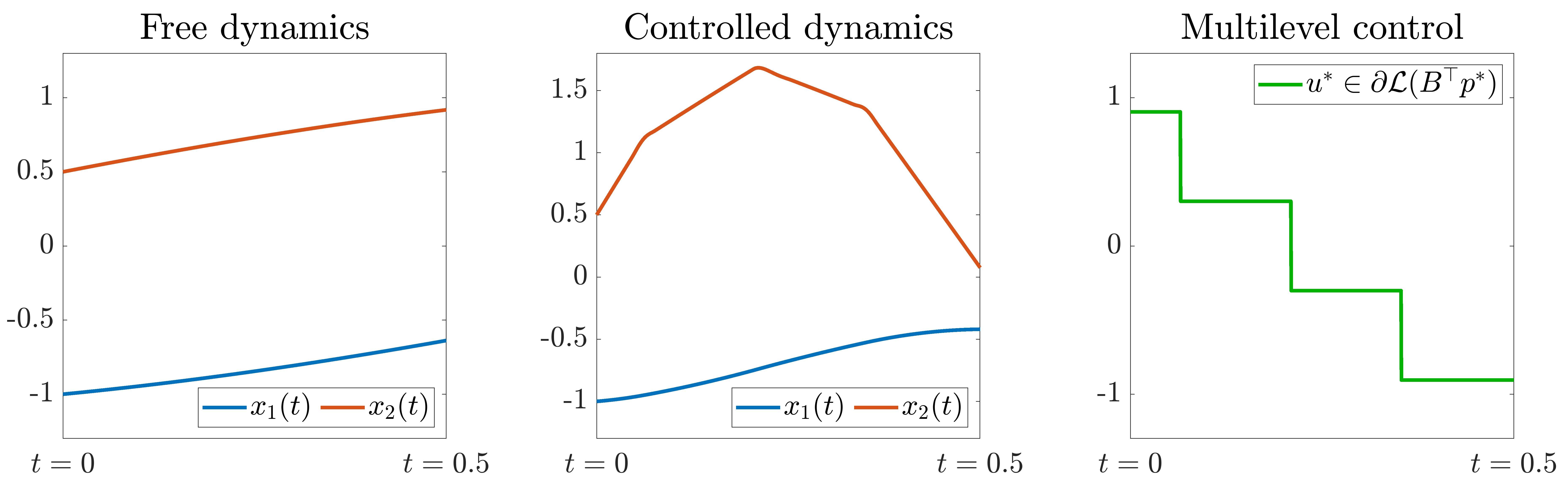}
	\caption{Free (left) and controlled (right) dynamics of the linear system \eqref{eq:oscillator} under the action of the multilevel control $u_{ml}^\ast$ computed via the minimization of the cost functional $J_{ml}$ on the time horizon $T=0.5$.}\label{fig:stateTsmall}
\end{figure}

In order to compute an effective multilevel control for \eqref{eq:oscillator} in the short time horizon $T=0.5$, we then have to employ the optimal control problem \eqref{eq:MultilevelFunctFabre} and minimize the functional $\mathcal J_{ml}$. The result of this minimization is shown in Figure \ref{fig:stateTsmallControlled}, where we can clearly see that, this time, the computed control is capable to steer the dynamics to zero at $T=0.5$.

\begin{figure}[h]
	\centering 
	\includegraphics[scale=0.38]{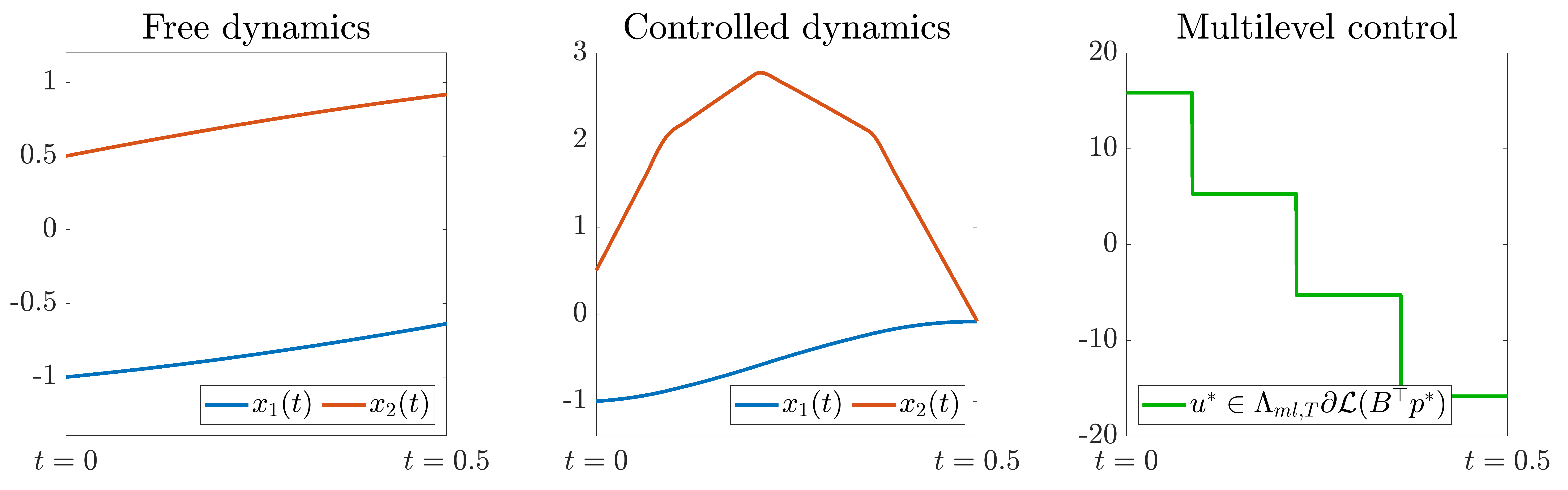}
	\caption{Free (left) and controlled (right) dynamics of the linear system \eqref{eq:oscillator} under the action of the multilevel control $u_{ml}^\ast$ computed via the minimization of the cost functional $\mathcal J_{ml}$ on the time horizon $T=0.5$.}\label{fig:stateTsmallControlled}
\end{figure}

Finally, let us recall that, as we have seen in the proof of Theorem \ref{thm:MLtheorem}, when considering a short time horizon the optimal control problem \eqref{eq:MultilevelFunct} for the functional $J_{ml}$ can still provide an effective multilevel control for solutions of \eqref{eq:oscillator} corresponding to small initial data. This is indeed observed in Figure \ref{fig:stateX0small}, where we display the dynamics of \eqref{eq:oscillator} in the time horizon $T=0.5$, corresponding to the initial datum $x_0 = (-0.25,0.25)^\top$ and the multilevel control obtained through the minimization of $J_{ml}$.

\begin{figure}[h]
	\centering 
	\includegraphics[scale=0.35]{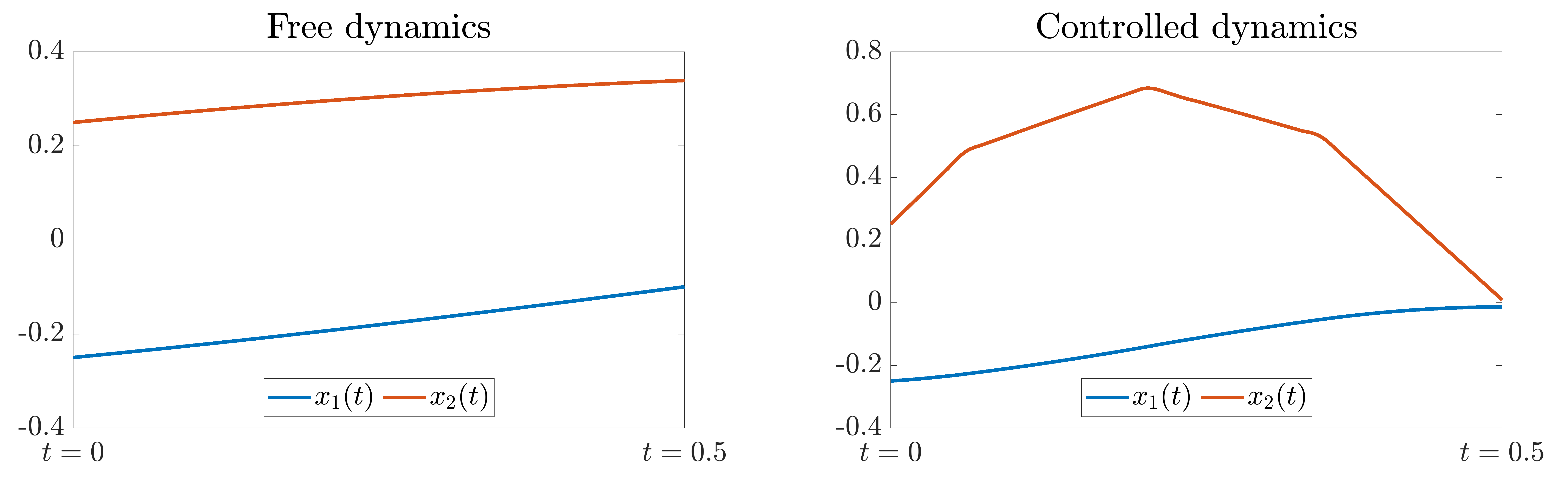}
	\caption{Free (left) and controlled (right) dynamics of the linear system \eqref{eq:oscillator} under the action of the multilevel control $u_{ml}^\ast$ computed via the minimization of the cost functional $\mathcal J_{ml}$ on the time horizon $T=0.5$.}\label{fig:stateX0small}
\end{figure}

Hence, in summary, our numerical simulations are consistent with the theoretical results we presented in Sections \ref{sec:problem}, \ref{sec:proof} and \ref{sec:structural}, thus confirming the validity of the adjoint methodology we have proposed. 

\section{Conclusions and open problems}\label{sec:conclusions}

In this paper, we have proposed an adjoint methodology to solve the multilevel control problem, which consists in generating piece-wise constant controls taking value in a finite-dimensional set and capable of steering the solution of a given linear finite-dimensional system from any initial datum in $x_0\in\RR^N$ to zero in time $T$.

More precisely, we have shown how these multilevel controls can be obtained via a dual optimization argument, which also allows to characterize some structural properties such as the minimal controllability time or the solvable set. 

In conclusion, the present paper gives a complete panorama on the multilevel control problem and how it can be efficiently solved. However, some relevant issues are not completely covered by our study, and will be considered in future works.

\begin{itemize}
	\item[1.] \textbf{Minimal number of switches in the multilevel control}. Our main results Theorem \ref{thm:MLtheorem} and \ref{thm:MLtheorem2} show that our proposed adjoint methodology is capable to generate multilevel controls for \eqref{eq:primalSystem}. Nevertheless, in practical applications, it may be important to keep track of the number of switches in the multilevel control and keeping it the lowest possible. It then becomes very relevant to determine which is the minimum number of switches in the multilevel control allowing to steer the solution of \eqref{eq:primalSystem} to zero in a given time horizon $T$. Notice that this is not at all a trivial question. From our characterization of multilevel controls (see \eqref{eq:controlMl}), we know that these switches arise in the points where the adjoint dynamics matches one of the values in the finite-dimensional set $\mathcal U$ defined in \eqref{eq:partitionU}. This dynamics being analytic, we know that the number of switches is finite. Nevertheless, to exactly determine this number is a much more difficult question which deserve a deeper investigation.
	
	\item[2.] \textbf{Complete characterization of the solvable set}. In Section \ref{sec:structural} we have given a characterization of the solvable set for the multilevel control problem. Nevertheless, what we have provided is actually a necessary conditions on the initial data, expressed in terms of some upper bound on their norm, implying that the solvable set is contained in some ball in the euclidean space $\RR^N$. Instead, it would be nice to obtain some sharper characterization of this solvable set and, possibly, its geometry.
\end{itemize}

\appendix
\section{Technical results} 

We collect here some technical results that we have employed in our proofs. We begin by showing the convergence of the piece-wise linear penalization $\mathcal L$ introduced in our multilevel optimal control problem \eqref{eq:MultilevelFunct} to the convex function $\mathcal P$ that this penalization interpolates. In particular, we have the following result.

\begin{lemma}\label{lem:convergence}
Let $\varpi>0$ and $\mathcal U$ be a partition of the interval $[-\varpi,\varpi]$ defined as
\begin{align*}
	&\mathcal U = \{u_1,\ldots,u_{M+1}\}, \quad M\geq 2
	\\
	&u_1 = -\varpi, \;u_{M+1} = \varpi \mbox{ and } u_k<u_{k+1}, \;\mbox{ for all } k\in \{1,\ldots,M\},
\end{align*}
with
\begin{align*}
	h_k:= u_{k+1}-u_k\;\mbox{ for all } k\in\{1,\ldots,M\}\quad\quad\mbox{ and }\quad\quad h:=\max_{k\in\{1,\ldots,M\}}h_k
\end{align*}

Let $\mathcal P\in C^2([-\varpi,\varpi])$ be a given strictly convex function and define the piece-wise linear interpolation of $\mathcal P$ on the partition $\mathcal U$ as 
\begin{align*}
	\mathcal L(u):= \begin{cases} \lambda_k(u) & \mbox{ if }u\in [u_k,u_{k+1}), \quad k\in\{1,\ldots,M\} \\ \mathcal P(u_{M+1}), & \mbox{ if }u=u_{M+1} \end{cases},	
\end{align*}
where 
\begin{align*}
	\lambda_k(u):= \frac{(u-u_k)\mathcal P(u_{k+1})+(u_{k+1}-u)\mathcal P(u_k)}{u_{k+1}-u_k}.
\end{align*}
Let $e_k$ and $e_{max}$ denote the local and global interpolation errors defined as 
\begin{align*}
	e_k:= \max_{u\in [u_k,u_{k+1})} |\mathcal P(u)-\lambda_k(u)|\;\mbox{ for all } k\in\{1,\ldots,M\} \quad\quad\mbox{ and }\quad\quad e_{max}:=\max_{k\in\{1,\ldots,M\}} e_k.
\end{align*}
Then, we have
\begin{align}\label{eq:errorInterpLocal}
	e_k\leq \frac{h_k^2}{2}\max_{u\in [u_k,u_{k+1})} |\mathcal P''(u)|\;\mbox{ for all } k\in\{i,\ldots,M\}
\end{align}
and
\begin{align}\label{eq:errorInterpGlobal}
	e_{max}\leq \frac{h^2}{2}\max_{u\in[-\varpi,\varpi]} |\mathcal P''(u)|.
\end{align}
In particular, as $M\to +\infty$, $\mathcal L\to\mathcal P$ a.e. in $[-\varpi,\varpi]$.
\end{lemma}

\begin{proof}
First of all, notice that, since the function $\mathcal P$ and its interpolant $\mathcal L$ coincide on the interpolation points, i.e. $\mathcal P(u_k)=\mathcal L(u_k)$ for all $k\in\{1,\ldots,M\}$, the fundamental theorem of calculus yields that
\begin{align*}
	\mathcal P(u)-\lambda_k(u) = \int_{u_k}^u (\mathcal P-\lambda_k)'(s)\,ds, \quad\mbox{ for all } u\in [u_k,u_{k+1}).
\end{align*}
Moreover, by the mean value theorem there exists some $\xi\in [u_k,u_{k+1})$ such that 
\begin{align*}
	\mathcal P'(\xi) = \frac{\mathcal P(u_{k+1})-\mathcal P(u_k)}{u_{k+1}-u_k} = \lambda_k'(\xi)\quad\longrightarrow\quad \mathcal (\mathcal P-\lambda_k)'(\xi) = 0.
\end{align*} 
In view of this, and using again the fundamental theorem of calculus, we then obtain 
\begin{align*}
	\mathcal P(u)-\lambda_k(u) = \int_{u_k}^u\int_\xi^s (\mathcal P-\lambda_k)''(\tau)\,d\tau ds, \quad\mbox{ for all } u\in [u_k,u_{k+1}).
\end{align*}
	
Taking into account that $\lambda_k$ is a linear function, i.e. $\lambda_k'' = 0$ a.e., we finally get from the previous identity that for all $u\in [u_k,u_{k+1})$
\begin{align}\label{eq:errorIneq}
	\mathcal P(u)-\lambda_k(u) &= \int_{u_k}^u\int_\xi^s \mathcal P''(\tau)\,d\tau ds \leq \max_{u\in[u_k,u_{k+1})}|\mathcal P''(u)|\int_{u_k}^u\int_\xi^s d\tau ds \notag 
	\\
	&= \max_{u\in[u_k,u_{k+1})}|\mathcal P''(u)|\frac{(u-u_k)(s-\xi)}{2}
	\\
	&\leq \frac{h_k^2}{2} \max_{u\in[u_k,u_{k+1})}|\mathcal P''(u)|. \notag
\end{align}
	
Then, the bounds \eqref{eq:errorInterpLocal} and \eqref{eq:errorInterpGlobal} for the local and global interpolation error follow immediately from \eqref{eq:errorIneq}. In particular, since $h\to 0$ as $M\to +\infty$, we also have
\begin{align*}
	|\mathcal P(u) - \mathcal L(u)|\to 0, \quad\mbox{ as } M\to +\infty.
\end{align*}
Our proof is then concluded.
\end{proof}

The following lemma shows how the sub-differential interacts with the convex conjugate of a given convex function. This result is actually very classical. Nevertheless, for the sake of completeness, we include its proof. 
\begin{lemma}\label{lem:subdifferential}
Let $f:\RR\to\RR$ be a convex function and $f^\star$ be its convex conjugate. Let $\partial f$ denote the sub-differential of $f$. Then, for all $u,v\in\RR$, we have that
\begin{align}\label{eq:subdiffEq}
	v\in\partial f(u) \Leftrightarrow u\in\partial f^\star(u).
\end{align}
\end{lemma}
	
\begin{proof}
First of all, we can easily see that 
\begin{align*}
	v\in \partial f(u) &\Leftrightarrow 0 \in \partial(f(u)-uv)
	\\
	&\Leftrightarrow u \in \underset{w\in\RR}{\mbox{argmin}} (f(w)-vw)
	\\
	&\Leftrightarrow u \in \underset{w\in\RR}{\mbox{argmax}} (vw-f(w))
	\\
	&\Leftrightarrow uv - f(u) = \max_{w\in\RR} (vw - f(w)) = f^\star(v)
\end{align*}
In particular, we have that
\begin{align*}
	uv - f^\star(v) = f(u).
\end{align*}
Then, since for $f$ convex we have $f^{\star\star} = f$, we obtain from the above computations that
\begin{align*}
	uv - f^\star(v) = \max_{w\in\RR} (uw-f^\star(w))
\end{align*}
which yields
\begin{align*}
	v\in \underset{w\in\RR}{\mbox{argmax}} (uw-f^\star(w)) &\Leftrightarrow v\in \underset{w\in\RR}{\mbox{argmin}} (f^\star(w)-uw)
	\\
	&\Leftrightarrow 0 \in \partial(f^\star(v)-uv)
	\\
	&\Leftrightarrow u \in \partial f^\star(v).
\end{align*}
This concludes our proof.
\end{proof}

Finally, we include below a couple of technical lemmas on the application of convex conjugates to piece-wise linear functions. In particular, we have the following results.

\begin{lemma}\label{lem:piecewiseChar}
Let $\mathcal I\subseteq\RR$ and $(\mathcal I_k)_{k=1}^M$ be a partition of $\mathcal I$, that is
\begin{align*}
	\mathcal I = \bigcup_{k\in\{1,\ldots,M\}} \mathcal I_k \quad\quad\mbox{ with }\quad\quad \begin{cases} \intr{\mathcal I_k}\neq\emptyset & \mbox{ for all } k\in\{1,\ldots,M\} \\ \intr{\mathcal I_k}\cap \intr{\mathcal I_\ell} = \emptyset & \mbox{ if } k\neq\ell \end{cases}.
\end{align*}
Let $f:\mathcal I\to\RR$ be a piece-wise linear and convex function such that
\begin{align*}
	f(u) = a_k u + b_k, \quad \mbox{ if } u\in\mathcal I_k.
\end{align*}
Then, $f$ can be characterized as
\begin{align}\label{eq:piecewiseChar}
	f(u) = \max_{k\in \{1,\ldots,M\}} \Big(a_ku + b_k \Big).
\end{align}
\end{lemma}

\begin{proof}
Since $f$ is convex, by Jensen's inequality we have that, for all $u,v\in\mathcal I$ and $t\in (0,1)$,
\begin{align*}
	f(v+t(u-v)) \leq f(v) + t(f(u)-f(v)), 
\end{align*} 
and hence
\begin{align}\label{eq:fIneq}
	f(u) \geq f(v) + \frac{f(v+t(u-v))-f(v)}{t}.
\end{align}
Suppose that $u\in \mathcal I_k$ for some $k\in\{1,\ldots,M\}$, and take $v\in \intr{\mathcal I_\ell}$ ($\ell\neq k$) and $t$ sufficiently small so that 
\begin{align*}
	v+t(u-v)\in\mathcal I_\ell. 
\end{align*}
Then, \eqref{eq:fIneq} reduces to
\begin{align*}
	a_k u + b_k \geq a_\ell v + b_\ell + \frac{a_\ell(v+t(u-v)) + b_\ell - a_\ell v - b_\ell}{t} = a_\ell u + b_\ell.
\end{align*}
Since this is true for all $\ell$, we conclude that
\begin{align*}
	a_k u + b_k = \max_{\ell\in\{1,\ldots,M\}} \Big(a_\ell u + b_\ell \Big),
\end{align*}
which is clearly equivalent to \eqref{eq:piecewiseChar}.
\end{proof}

\begin{lemma}\label{lem:piecewiseConj}
Let $\mathcal I\subseteq\RR$ and $f:\mathcal I\to\RR$ be a piece-wise linear and convex function. Let
\begin{align*}
	f^\star(v) = \sup_u\Big(uv-f(u)\Big) 
\end{align*}
denote the convex conjugate of $f$.	Then, $f^\star$ is a piece-wise linear and convex function on $\RR$. 
\end{lemma}

\begin{proof}
First of all we have that, according to Lemma \ref{lem:piecewiseChar}, the function $f$ can be written as 
\begin{align*}
	f(u) = \max_{k\in \{1,\ldots,M\}} \Big(a_ku + b_k \Big).
\end{align*}

Moreover, without losing generality we can assume that the coefficients $a_k$ are sorted in increasing order, i.e. $a_1\leq a_2\leq\ldots\leq a_M$, and that none of the functions $a_ku + b_k$ is redundant, i.e. for each $k\in\{1,\ldots,M\}$ there exists at least one $u\in\mathcal I$ with $f(u) = a_ku + b_k$. Under this assumption, we have that the graph of $f$ is piece-wise linear with break-points (see \cite[Chapter 3]{boyd2004convex} and \cite[Section 3.31]{boyd2004solutions})
\begin{align*}
	\frac{b_k-b_{k+1}}{a_{k+1}-a_k}, \quad \mbox{ for all } k\in\{1,\ldots,M\}.
\end{align*}
By definition of convex conjugate, we have that
\begin{align}\label{eq:fConjDef}
	f^\star(v) = \sup_u\left(uv-\max_{k\in \{1,\ldots,M\}} \Big(a_ku + b_k \Big)\right). 
\end{align}

We see that the domain of $f^\star$ is the closed interval $[a_1,a_M]$, since for $v$ outside that range the expression inside the supremum is unbounded above. Moreover, for $v\in[a_k,a_{k+1}]$, the supremum \eqref{eq:fConjDef} is reached at the break-point between the segments $k$ and $k+1$, i.e. 
\begin{align}\label{eq:supremum}
	v\in [a_k,a_{k+1}]\quad\longrightarrow\quad \sup_u\left(uv-\max_{k\in \{1,\ldots,M\}} \Big(a_ku + b_k \Big)\right) =  \frac{b_{k+1}-b_k}{a_{k+1}-a_k}, \quad \mbox{ for all } k\in\{1,\ldots,M\}.
\end{align}
Using \eqref{eq:supremum}, we obtain that 
\begin{align*}
	f^\star(v) = -b_k -(b_{k+1}-b_k)\frac{v-a_k}{a_{k+1}-a_k}.
\end{align*}
In particular, we have that $f^\star$ is a piece-wise linear function on $\RR$:
\begin{align*}
	f^\star(v) = \alpha_k v + \beta_k \quad \mbox{ if } v\in [a_k,a_{k+1}], \quad\mbox{ with } \alpha_k:=-\frac{b_{k+1}-b_k}{a_{k+1}-a_k} \;\mbox{ and }\;\beta_k:= \frac{a_k(b_{k+1}-b_k)}{a_{k+1}-a_k}-b_k.
\end{align*}
\end{proof}

\bibliographystyle{acm}
\bibliography{biblio}

\end{document}